\documentclass[11pt,authoryear,letterpaper,reqno,oneside]{amsart}
\pdfoutput=1

\usepackage[left=1.5in,right=1.5in,bottom=1.2in,top=1.2in]{geometry}

\usepackage{amsmath}
\usepackage{amssymb,amsthm}
\usepackage{wasysym}
\usepackage{stmaryrd}
\usepackage{microtype}

\usepackage[normalem]{ulem}
\usepackage[usenames]{color}
\usepackage{mathrsfs}
\usepackage{graphicx}\graphicspath{{diagram/}}

\definecolor{LightGray}{rgb}{.6,.6,.6}
\definecolor{darkred}{rgb}{0.5,0,0}
\definecolor{darkgreen}{rgb}{0, 0.3,0}
\definecolor{darkblue}{rgb}{0,0,0.6}

\usepackage[colorlinks,citecolor=darkblue,urlcolor=darkblue,linkcolor=darkgreen]{hyperref}
\usepackage{hypernat}

\theoremstyle{plain}
\newtheorem{proposition}{Proposition}[section]
\newtheorem{lemma}[proposition]{Lemma}
\newtheorem{corollary}[proposition]{Corollary}

\newtheorem{theorem}[proposition]{Theorem}

\theoremstyle{definition}
\newtheorem{definition}[proposition]{Definition}

\theoremstyle{remark}
\newtheorem{construction}[proposition]{Construction}

\newcommand{\llrr}[1]{{\llbracket #1 \rrbracket}}
\newcommand{\bigllrr}[1]{{\bigl \llbracket #1 \bigr \rrbracket}}

\newcommand{\QU}{{\mathbb{Q}\mathbb{U}}}
\newcommand{\plusdot}{{\reflectbox{\rotatebox[origin=c]{180}{$\dotplus$}}}}
\newcommand{\defn}[1]{{\bf{#1}}}
\newcommand{\half}{{\frac{1}{2}}}

\newcommand{\Ohm}{{\mathrm \Omega}}

\newcommand{\Lwow}{\ensuremath{\cL_{\omega_1,\omega}}}
\newcommand{\Lww}{\ensuremath{\cL_{\omega,\omega}}}
\def\Str{{\mathrm{Str}}}
\newcommand{\Model}{\ensuremath{{\Str_L}}}
\newcommand{\Fr}{Fra\"iss\'e}

\DeclareMathOperator{\Sub}{Clo}

\newcommand{\st}{{\ :\ }}

\newcommand{\cF}{{\mathcal{F}}}
\newcommand{\cR}{{\mathcal{R}}}
\newcommand{\cC}{{\mathcal{C}}}
\newcommand{\ccD}{{\mathcal{D}}}

\renewcommand{\Pr}{{\mathbb{P}}}

\newcommand{\M}{{\mathcal{M}}}
\newcommand{\cM}{{\mathcal{M}}}
\newcommand{\cK}{{\mathcal{K}}}
\newcommand{\XX}{{\mathcal{N}}}
\newcommand{\cN}{{\mathcal{N}}}
\newcommand{\cA}{{\mathcal{A}}}
\newcommand{\Abar}{{\overline{\mathcal{A}}}}
\newcommand{\Mbar}{{\overline{\mathcal{M}}}}
\newcommand{\Nbar}{{\overline{\mathcal{N}}}}
\newcommand{\cB}{{\mathcal{B}}}
\newcommand{\Bbar}{{\overline{\mathcal{B}}}}

\newcommand{\cL}{{\mathcal{L}}}

\newcommand{\PP}{{\mathcal{P}}}
\newcommand{\PPbar}{{\overline{\PP}}}
\newcommand{\QQ}{{\mathcal{Q}}}
\newcommand{\T}{{\mathcal{T}}}

\newcommand{\F}{{\mathcal{F}}}
\newcommand{\B}{{\mathcal{B}}}
\newcommand{\X}{{\mathcal{X}}}

\newcommand{\G}{{\mathcal{G}}}
\newcommand{\GG}{{\mathbb{G}}}
\renewcommand{\L}{\\L}

\newcommand{\pt}{p_{2i+1}}

\renewcommand{\aa}{{\mathbf{a}}}

\renewcommand{\AA}{{\mathbf{A}}}
\newcommand{\bb}{{\mathbf{b}}}
\newcommand{\cc}{{\mathbf{c}}}
\newcommand{\dd}{{\mathbf{d}}}

\newcommand{\mm}{{\mathbf{m}}}

\newcommand{\rr}{{\mathbf{r}}}

\newcommand{\vv}{{\mathbf{v}}}
\newcommand{\xx}{{\mathbf{x}}}
\newcommand{\uu}{{\mathbf{u}}}
\newcommand{\yy}{{\mathbf{y}}}
\newcommand{\zz}{{\mathbf{z}}}
\newcommand{\ww}{{\mathbf{w}}}

\renewcommand{\And}{\wedge}
\newcommand{\Or}{\vee}

\newcommand{\Naturals}{\mathbb{N}}
\newcommand{\Nats}{\Naturals}

\newcommand{\defas}{:=}

\DeclareMathOperator{\Aut}{Aut}
\DeclareMathOperator{\dcl}{dcl}
\DeclareMathOperator{\acl}{acl}
\newcommand{\sym}{{S_\infty}}

\newcommand{\Rado}{{\mathscr{R}}}
\newcommand{\Henson}{{\mathscr{H}_3}}

\newcommand{\Reals}{\ensuremath{\mathbb{R}}}

\newcommand{\Rationals}{\ensuremath{\mathbb{Q}}}

\newcommand{\w}{\ensuremath{\omega}}

\begin{document}

\title{Invariant measures concentrated \linebreak on countable structures}

\author[Ackerman]{Nathanael Ackerman}
\address{
Department of Mathematics\\
Harvard University\\
One Oxford Street\\
Cambridge, MA 02138}
\email{nate@math.harvard.edu}
\thanks{}

\author[Freer]{Cameron Freer}
\address{
Computer Science and Artificial Intelligence Laboratory\\
Massachusetts Institute of Technology\\
32 Vassar Street\\
Cambridge, MA 02139
}
\email{freer@mit.edu}
\thanks{}

\author[Patel]{Rehana Patel}
\address{
Franklin W.\ Olin College of Engineering\\
1000 Olin Way\\
Needham, MA 02492
}
\email{rehana.patel@olin.edu}
\thanks{}

{\let\thefootnote\relax\footnotetext{
2010~{\it Mathematics Subject Classification.}
Primary:
03C98;
Secondary:
60G09,
37L40,
05C80,
03C75,
62E10,
05C63.
}

\maketitle

\thispagestyle{empty}

\vspace*{-20pt}


\begin{abstract}
Let $L$ be a countable language. We say that a countable infinite $L$-structure
$\cM$ admits an invariant measure when there is a probability measure on
the space of $L$-structures with the same underlying set as $\cM$ that is invariant
under permutations of that set, and that assigns measure one to the
isomorphism class of $\cM$. We show that $\cM$ admits an invariant measure if and
only if it has trivial definable closure, i.e., the pointwise stabilizer in
$\Aut(\cM)$ of an arbitrary finite tuple of $\cM$ fixes no additional points. 
When $\cM$ is a \Fr\ limit in a relational language, this amounts to requiring that the age of $\cM$ have strong amalgamation.
Our results give rise to new instances of structures that admit invariant
measures and structures that do not. 
\end{abstract}


\renewcommand\contentsname{\!\!\!\!}
\setcounter{tocdepth}{2}
\vspace*{5pt}
{\scriptsize
\tableofcontents
}
\ \\
\newpage

\section{Introduction}
\label{intro-sec}

Randomness is used to construct objects throughout mathematics, and structures resulting from symmetric random constructions often exhibit structural regularities.  Here we characterize, in terms of a combinatorial criterion, those countable structures in a countable language that can be built via a random construction that is invariant
under
reorderings of the elements.

A probabilistic construction is \emph{exchangeable} when its distribution satisfies the symmetry condition of being invariant 
under
permutations of its elements.  When an exchangeable construction almost surely produces a single structure (up to isomorphism), we say that the structure \emph{admits an invariant measure}.  Such structures often exhibit regularity properties such as universality or ultrahomogeneity.
Two of the 
most important
randomly constructed structures 
with these regularities
are Rado's countable universal ultrahomogeneous graph and Urysohn's universal separable ultrahomogeneous metric space.  The Rado graph may be generated as a random graph by independently choosing edges according to the Erd{\H{o}}s--R\'enyi construction \cite{MR0120167}, and Urysohn space arises 
via (the completion of) an exchangeable countable metric space, by a construction of
Vershik 
\cite{MR2006015}, \cite{MR2086637}.

Because these examples have such rich internal structure, it is natural to ask 
which
other objects admit invariant measures.  
One formulation of this question was posed by Cameron in \cite[\S4.10]{MR1066691}.
Petrov and Vershik \cite{MR2724668} have recently shown, using
a new type of construction, 
that the countable universal ultrahomogeneous $K_n$-free graphs all admit invariant measures.  In the present work, we combine methods from the model theory of infinitary logic with ideas from Petrov and Vershik's construction to give a complete characterization of those countable infinite structures in a countable language that admit invariant measures. Specifically, we show that a structure $\cM$ admits an invariant measure if and only if the pointwise stabilizer in $\Aut(\cM)$ of any finite set of elements of $\cM$ fixes no additional elements, a condition known as having \emph{trivial definable closure}.

Many  natural examples of objects admitting invariant measures are generic, in the sense of being \Fr\ limits, i.e., the countable universal ultrahomogeneous object for some class of finite structures \cite[\S7.1]{MR1221741}.  One may ask what additional regularity properties must hold of \Fr\ limits that admit invariant measures.  \Fr\ limits arise from amalgamation procedures for ``gluing together'' finite substructures.  Our result implies that a \Fr\ limit in a countable relational language admits an invariant measure if and only if it has \emph{strong amalgamation}, a natural restriction on the gluing procedure that produces the limit.

Our characterization gives rise to new 
instances
of structures 
that admit
invariant measures, and 
structures
that do not.  We apply our results to existing classifications of ultrahomogeneous graphs, directed graphs, and partial orders, as well as other combinatorial structures, thereby providing several new examples of exchangeable constructions that lead to generic structures.  
Among 
those structures
for which we provide the first such constructions are
the countable universal ultrahomogeneous partial order \cite{MR544855} and
certain countable universal graphs forbidding a finite homomorphism-closed
set of finite connected graphs \cite{MR1683298}.
Structures for which our results imply the non-existence of such constructions
include Hall's countable universal group \cite[\S7.1, Example~1]{MR1221741}, 
and 
the existentially complete countable universal bowtie-free graph \cite{MR1675931}.

\subsection{Background}\ \\
\label{background}
\indent The Rado graph $\Rado$, sometimes known as the ``random graph'', is
(up to isomorphism) the unique countable universal ultrahomogeneous graph
\cite{MR0172268}.  It is the Fra\"iss\'e limit of the class of finite
graphs, with a first-order theory characterized by so-called ``extension
axioms" that have a simple syntactic form.  It is  also the classic example
of a countable structure that has a symmetric probabilistic construction,
namely, the countably infinite version of the Erd{\H{o}}s--R\'enyi random
graph process introduced by Gilbert \cite{MR0108839} and Erd{\H{o}}s and
R\'enyi \cite{MR0120167}.  For $0<p<1$, this process determines a random
graph on a countably infinite set of vertices by independently flipping a
coin of weight $p$ for every pair of distinct vertices, and adding an edge
between those vertices precisely when the coin flip comes up heads.  Denote
this random variable by $\GG(\Nats, p)$.  The random graph $\GG(\Nats,p)$ is
almost surely  isomorphic to $\Rado$, for any $p$ such that $0< p <1$.
Moreover, each $\GG(\Nats, p)$ is \emph{exchangeable}, i.e., its distribution is invariant under arbitrary permutations of the vertices, and so there are continuum-many different invariant measures concentrated on $\Rado$ (up to isomorphism).  It is natural to ask which other structures admit random constructions that are invariant in this way.

Consider the Henson graph $\Henson$, the unique (up to isomorphism) countable universal ultrahomogeneous triangle-free graph \cite{MR0304242}.  Like the Rado graph, it has a first-order theory consisting of extension axioms, and can be constructed as a \Fr\ limit. Does it also admit an invariant measure?  In contrast with $\Rado$, no countable random graph whose distribution of edges is independent and identically distributed (i.i.d.)\ can be almost surely isomorphic to $\Henson$.  But this does not rule out the possibility of an exchangeable random graph almost surely isomorphic to $\Henson$. Its distribution would constitute a measure on countable graphs, invariant under arbitrary permutations of the underlying vertex set, that is concentrated on 
the isomorphism class of
$\Henson$.

One might consider building an invariant measure concentrated on 
(graphs isomorphic to)
$\Henson$
by ``approximation
from below'' using uniform measures on finite
triangle-free graphs, by analogy with the invariant measure concentrated on
$\Rado$ obtained as the weak limit of uniform measures on finite graphs.
The distribution of finite Erd{\H{o}}s--R\'enyi random graphs $\GG(n, \half)$
is simply the uniform measure on graphs with $n$ labeled vertices; the
sequence $\GG(n,\half)$ converges in distribution to $\GG(\Nats, \half)$, which is almost surely isomorphic to $\Rado$.  So to obtain an invariant measure concentrated on $\Henson$, one might similarly consider the weak limit of  the sequence of uniform measures on finite triangle-free labeled graphs of size $n$, i.e., of the distributions of the random graphs $\GG(n,\half)$ conditioned on being triangle-free.  However, by work of Erd{\H{o}}s, Kleitman, and Rothschild \cite{MR0463020} and Kolaitis, Pr\"omel, and Rothschild \cite{MR902790}, this sequence is asymptotically almost surely bipartite, and so its weak limit is almost surely not isomorphic to $\Henson$. Hence, as noted in \cite{MR2724668}, this particular approach does not produce an invariant measure concentrated on $\Henson$.

Petrov and Vershik \cite{MR2724668} provided the first instance of an invariant measure concentrated on the Henson graph $\Henson$ (up to isomorphism); they also did likewise for Henson's other countable universal ultrahomogeneous $K_n$-free graphs, where $n > 3$.  They produced this measure via a ``top down" construction, building a continuum-sized triangle-free graph in such a way that an i.i.d.\ sample from its vertex set induces an exchangeable random graph that is almost surely isomorphic to $\Henson$.

In this paper, we address the question of invariant measures concentrated on \emph{arbitrary} structures. Given a countable language $L$ and a countable infinite $L$-structure $\cM$, we ask whether  there exists a probability measure 
on the space of $L$-structures with the same underlying set as $\cM$,
invariant under arbitrary permutations of 
that set, assigning
measure one to the isomorphism class of $\cM$.  We provide a complete answer to this question, by characterizing such $L$-structures $\cM$ as precisely those that have trivial group-theoretic definable closure, i.e.,  those structures $\cM$ for which the pointwise stabilizer  in Aut($\cM$) of any finite tuple $\aa$ from $\cM$ fixes no elements of $\cM$ except those in $\aa$.  We use infinitary logic  to establish a setting in which, whenever $\cM$ has trivial definable closure, we can
construct continuum-sized objects that upon sampling 
produce
invariant measures concentrated on $\cM$.  When $\cM$ does not have trivial definable closure, we show that such invariant measures cannot exist.

Our results build on several ideas from \cite{MR2724668}. In particular, Petrov
and Vershik show that if a continuum-sized graph satisfies certain
properties, then sampling from 
it
produces an invariant measure
concentrated on $\Henson$ (and similar results for $\Rado$ and the other Henson $K_n$-free graphs);
they then proceed to construct such continuum-sized graphs.
We identify a certain type of continuum-sized structure 
whose existence guarantees,
using
a similar sampling procedure, an invariant measure concentrated on a target countable structure; we then construct such a continuum-sized structure whenever the target structure has trivial definable closure.

Underlying Petrov and Vershik's 
construction
of invariant measures,
as well as ours,
is
the characterization of countable exchangeable (hyper)graphs as those obtained via
certain sampling procedures from continuum-sized structures. These ideas
were developed by
Aldous \cite{MR637937}, Hoover \cite{Hoover79}, Kallenberg \cite{MR1182678} 
and Vershik \cite{MR1922015}  in work on the probability theory of
exchangeable arrays.
More recently, similar machinery
has come to prominence in the combinatorial theory of limits of dense
graphs via \emph{graphons}, due to Lov\'asz and Szegedy \cite{MR2274085}
and others.  For an equivalence between these characterizations, see Austin
\cite{MR2426176} and Diaconis and Janson \cite{MR2463439}. 
The \emph{standard recipe} of \cite{MR2426176} provides a more general formulation of the correspondence between sampling procedures on continuum-sized objects and arbitrary countable exchangeable structures.

In the present paper, we are interested primarily in determining those countable infinite structures for which there exists at least one invariant measure concentrated on its isomorphism class.  In the case of countable graphs, our construction in fact provides a new method for building graphons.  In particular, the graphons we build are \emph{random-free}, in the sense of 
\cite[\S10]{MR3043217}.  Therefore our construction shows that whenever there is an invariant measure concentrated on the isomorphism class of a countable graph, there is such a measure that comes from sampling a random-free graphon.

Within mathematical logic, the study of invariant measures on countable
first-order structures goes back to work of Gaifman \cite{MR0175755}, Scott and Krauss \cite{ScottKrauss}, and 
Krauss \cite{MR0275482}. For a discussion of this earlier history and
its relationship to Hoover's work on exchangeability, see Austin
\cite[\S3.8 and \S4.3]{MR2426176}.
Our countable relational setting is akin to that explored more
recently in extremal combinatorics by Razborov \cite{MR2371204}; for
details see
\cite[\S4.3]{MR2426176} and \cite{2008arXiv0801.1538A}.

Other work in model theory has examined aspects of
probabilistic constructions.
Droste and Kuske \cite{MR1979531} and Dolinka and Ma{\v{s}}ulovi{\'c}
\cite{MR2891379} describe probabilistic
constructions of countable infinite structures, though without invariance.
Usvyatsov \cite{MR2435152} has also considered a relationship between invariant measures and notions of genericity in the setting of continuous first-order logic, especially with respect to Urysohn space.

\subsection{Main results}\ \\
\indent Our main theorem characterizes countable infinite structures $\cM$ that admit invariant measures as those for which the pointwise stabilizer, in $\Aut(\cM)$, of an arbitrary finite tuple of $\cM$ fixes no additional points.  For a countable language $L$, let $\Model$ be the Borel measure space of $L$-structures with underlying set $\Naturals$. 
(This is a standard space on which to consider measures invariant under the action of the permutation group $\sym$; we
provide details in \S\ref{logicaction}.) Then we have the following result.

\begin{theorem}\label{maintheorem}
Let $L$ be a countable language, and let $\cM$ be a countable infinite
$L$-structure.
The following are equivalent:
\begin{itemize}
\item[(1)]
There is a probability measure on
$\Model$, invariant under the natural action of
$\sym$ on $\Model$,
that is concentrated on 
the set of elements of $\Model$ that are isomorphic to
$\cM$.
\item[(2)]
The structure $\cM$
has trivial
group-theoretic definable closure, 
i.e.,
for every finite tuple $\aa\in\cM$, we have
$\dcl_{\cM}(\aa)= \aa$, where
$\dcl_{\cM}(\aa)$ is the collection of elements of $\cM$ that are fixed by
all automorphisms of $\cM$ fixing $\aa$ pointwise.
\end{itemize}
\end{theorem}

Note that every finite structure admits a natural probability measure that
is invariant under permutations of the underlying set. But also every
finite structure has nontrivial definable closure, and so the statement of
this theorem does not extend to finite structures.

Our main result is the equivalence of (1) and (2); 
but further,
an observation of Kechris and Marks shows the 
additional
equivalence
with 
(3) 
in Theorem~\ref{maintheoremthree}
below.

For any structure $\cN \in \Model$, we write $\Aut(\cN)$ for its automorphism group considered as a subgroup of $\sym$,
and take the action of $\Aut(\cN)$ on $\Model$ to be that given by the restriction of the natural action of $\sym$.

\begin{theorem}
\label{maintheoremthree}
Properties
{\rm(1)} and {\rm (2)} from
Theorem~\ref{maintheorem} are also equivalent to the following:
\begin{itemize}
\item[(3)] 
There is some $\cN\in\Model$ that has trivial group-theoretic definable closure and is such that
there is an
$\Aut(\cN)$-invariant
probability measure on
$\Model$
concentrated on the 
set of elements of $\Model$ that are isomorphic to
$\cM$.
\end{itemize}
\end{theorem}

Note that (3) is ostensibly weaker than (1), as 
in general
an $\Aut(\cN)$-invariant measure need not be $\sym$-invariant.

A structure $\cM$ is said to be \emph{ultrahomogeneous} when every partial isomorphism between finitely generated substructures of $\cM$ extends to an automorphism of $\cM$.  Define the \emph{age} of a countable $L$-structure $\cM$ to be the class of all finitely generated $L$-structures that are isomorphic to a substructure of $\cM$.  The age of any countable infinite ultrahomogeneous $L$-structure has the so-called \emph{amalgamation property}, which stipulates that any two structures in the age can be ``glued together'' over any common substructure, preserving this substructure but possibly identifying other elements. Countable infinite ultrahomogeneous $L$-structures have been characterized by \Fr\ as those obtained from their ages via a canonical ``back-and-forth'' construction using amalgamation; they are often called \Fr\ limits and are axiomatized by $\Pi_2$ ``extension axioms''.  (For details, see \cite[Theorems 7.1.4, 7.1.7]{MR1221741}.)

A standard result {\cite[Theorem~7.1.8]{MR1221741}} (see also \cite[\S2.7]{MR1066691}) states that when $\cM$ is a countable infinite ultrahomogeneous structure in a countable relational language, $\cM$ has trivial definable closure precisely when its age satisfies the more stringent condition known as the \emph{strong amalgamation property}, which requires that no elements (outside the intersection) are identified during the amalgamation.  Note that in {\cite[\S7.1]{MR1221741}}, strong amalgamation is shown to be equivalent to a property known as ``no algebraicity'', which is equivalent to our notion of 
(group-theoretic)
trivial definable closure for structures in a language with only relation symbols (but not for structures in a language with constant or function symbols).
Thus we obtain the following corollary to Theorem~\ref{maintheorem}.

\begin{corollary}
\label{maincorollary}
Let $L$ be a countable relational language, and let $\cM$ be a countable
infinite $L$-structure.
Suppose $\cM$ is ultrahomogeneous.
The following are equivalent:
\begin{itemize}
\item[(1)]
There is a probability measure on
$\Model$, invariant under the natural action of
$\sym$ on $\Model$,
that is concentrated on 
the set of elements of $\Model$ that are isomorphic to
$\cM$.
\item[($2'$)] The age of $\cM$ satisfies the strong amalgamation
property.
\end{itemize}
\end{corollary}

At the Workshop on Homogeneous Structures, held at the University of
Leeds in 2011, Anatoly Vershik asked whether an analogue of the notion of a
continuum-sized \emph{topologically universal graph} \cite{MR2724668}
exists for an arbitrary \Fr\ limit.
We propose our notion of a
(continuum-sized)
\emph{Borel $L$-structure strongly witnessing
a theory},
defined in Section~\ref{existence},
as an appropriate analogue. 

Our results then show that such a Borel $L$-structure can exist for a \Fr\ limit precisely when its age has the strong amalgamation property.  
If the age of a \Fr\ limit $\cM$ in a countable relational
language $L$
has the strong amalgamation property, then the proof of our main result
involves building a Borel $L$-structure that, just like a topologically
universal graph, has a ``large'' set of witnesses for every (nontrivial) extension axiom.  On the other hand, when the age of $\cM$ does not have the strong amalgamation property, such a Borel $L$-structure cannot exist; according to the machinery of our proof, it would necessarily induce an invariant measure concentrated on $\cM$, violating Corollary~\ref{maincorollary}.

\subsection{Outline of the paper}\ \\
\indent
We begin, in Section~\ref{prelim}, by describing our setting and providing preliminaries.
Throughout this paper we work in a countable language $L$.  We first describe the
infinitary language $\Lwow(L)$.  In particular, we recall the notion of a \emph{Scott
sentence}, a single infinitary sentence in $\Lwow(L)$ that describes a countable structure
up to isomorphism (among countable structures).  
We then define a certain kind of
infinitary $\Pi_2$ sentence, which we call \emph{pithy $\Pi_2$}, and which 
can be thought
of as a ``one-point'' extension axiom.  
We go on to
describe the measure space $\Model$ and define the natural action of $\sym$ on $\Model$,
called the \emph{logic action}. Using these notions, we explain what is meant by
an \emph{invariant measure} and what it means for a measure to be \emph{concentrated} on a
set of structures.  We then recall the group-theoretic notion of \emph{definable closure} and its
connection to the model theory of $\Lwow(L)$.  
Next, 
for any given countable $L$-structure $\cM$,
we describe its \emph{canonical language $L_\Mbar$} and \emph{canonical structure $\Mbar$};
the latter is essentially equivalent to $\cM$, but
is characterized (among
countable structures) by a theory $T_\Mbar$ consisting
entirely of 
``one-point'' extension axioms.
We show that $\cM$ admits an invariant measure if and only if $\Mbar$ does, and that $\cM$ has trivial definable closure if and only if $\Mbar$ does.
Finally,
we review some basic conventions from probability theory.

In Section~\ref{existence}, we prove the existence of an invariant
measure concentrated on an $L$-structure $\cM$ that has trivial definable
closure. We do so by constructing an invariant measure concentrated on its canonical structure $\Mbar$.

The invariant measures that we build in Section~\ref{existence} each come from sampling a continuum-sized
structure. Our method uses a similar framework to that employed 
by Petrov and Vershik in
\cite{MR2724668} for graphs.  
The first-order theory of the Henson graph $\Henson$ is generated by a set
of $\Pi_2$ axioms that characterize it up to isomorphism among countable
graphs.  
Petrov and Vershik construct an invariant
measure concentrated on $\Henson$ by building a continuum-sized structure
that realizes a ``large'' set of witnesses for each of these axioms.  

In our generalization of their construction, 
we build 
a continuum-sized structure that satisfies 
$T_\Mbar$
in a particularly strong way, analogously to \cite{MR2724668}.  Specifically, given 
a $\Pi_2$ sentence of the form $(\forall \xx)(\exists y) \psi(\xx,y)$ in 
$T_\Mbar$,
we ensure that for every tuple $\aa$ in the structure, the sentence $(\exists y)\psi(\aa,y)$ has a ``large'' set of witnesses, whenever $\psi(\aa, b)$ does not hold for any $b\in\aa$.  The construction proceeds inductively by  defining quantifier-free types on intervals, interleaving successive refinements of existing intervals with enlargements by new intervals that provide the ``large'' sets of witnesses.  This is possible by virtue of 
$T_\Mbar$
having a property we call \emph{duplication of quantifier-free types}, which occurs precisely when 
$\Mbar$
has trivial definable closure.  The continuum-sized structure built in this way is such that a random countable structure induced by sampling from it, with respect to an appropriate measure, will be a model of 
$T_\Mbar$
almost surely,  thereby producing 
an
invariant measure
concentrated on $\Mbar$.

Section~\ref{nonexistence} provides the converse, for an arbitrary countable language $L$: If a countable infinite $L$-structure has nontrivial definable closure, it cannot admit an invariant measure.  This is a direct argument that does not require the machinery developed in 
Section~\ref{existence}. 
In fact we present a generalization of
the converse, due to Kechris and Marks,
which states that such an $L$-structure cannot even admit an $\Aut(\cN)$-invariant measure for any $\cN\in\Model$ having trivial definable closure.

In Section~\ref{examples-nonexamples} we apply 
Theorem~\ref{maintheorem}
and 
Corollary~\ref{maincorollary}
to obtain examples of countable infinite structures that admit invariant measures, and those that do not. 
We describe how any structure can be ``blown up'' into one
that admits an invariant measure and also into one that does not.
This allows us to
give examples of countable structures having arbitrary Scott rank that admit invariant measures, and examples of those that do not admit invariant measures.
We then analyze definable closure in well-known countable structures to determine
whether or not they admit invariant measures. 

We conclude, in Section~\ref{more-general}, with several connections to the theory of
graph limits, and additional applications of our results.


\section{Preliminaries}\label{prelim}

Throughout this paper we use uppercase letters to represent sets, lowercase
letters to represent elements of a set and lowercase boldface letters to
represent finite tuples (of variables, or of elements of a structure).
The length $|\xx|$ of a tuple of variables $\xx$ is the number of entries,
not the number of distinct variables, in the tuple, and likewise for
tuples of elements.  We use the notation $(x_1,\ldots ,x_k)$ and $x_1\cdots
x_k$ interchangeably to denote a tuple of variables $\xx$ of length $k$
that has entries $x_1, \ldots , x_k$, in that order, and similarly for
tuples of elements. When it enhances clarity, we write, e.g., $(\xx, \yy)$
for $(x_1, x_2, y_1, y_2)$, when $\xx = x_1 x_2$ and $\yy = y_1 y_2$.
For an $n$-tuple $\aa \in A^n$,
we frequently abuse notation and write $\aa \in A$.

\subsection{Infinitary logic}  \ \\
\indent
We begin by reviewing some basic definitions  from logic. A \emph{language} $L$, also called a \emph{signature}, is a set $L \defas \cR \cup \cC \cup \cF$, where $\cR$ is a set of \emph{relation symbols}, $\cC$ is a set of \emph{constant symbols}, and $\cF$ is a set of \emph{function symbols}, all disjoint.  For each relation symbol $R\in\cR$ and function symbol $f\in\cF$, fix an associated positive integer, called its \emph{arity}.  We take the \emph{equality symbol}, written $=$, to be a logical symbol, not a binary relation symbol in $L$.  In this paper, 
all languages 
are
countable.
Given a language $L$, an \emph{$L$-structure} $\cM$ is a non-empty set $M$ endowed with interpretations of the symbols in $L$.
We sometimes write $x\in \mathcal{M}$ in place of $x \in M$. 

We now describe the class $\Lwow(L)$ of infinitary formulas in the language $L$.  
For more on infinitary logic and Scott sentences, see \cite{MR0344115}, \cite{MR0424560}, or \cite[\S2.4]{MR1924282}.
For the basics of first-order languages, terms, formulas, and theories, see \cite[\S1.1]{MR1924282}. 

\begin{definition}
The class $\Lwow(L)$ is the smallest collection of formulas that 
contains
all atomic formulas of $L$;
the formulas $(\exists x)\psi(x)$ and $\neg \chi$, where $\psi(x), \chi \in \Lwow(L)$; and
the formula $\bigwedge_{i\in I} \varphi_i$, where $I$ is an arbitrary countable set, 
$\varphi_i \in \Lwow(L)$
for each $i\in I$, and the set of free variables of $\bigwedge_{i \in I} \varphi_i$ is finite.
\end{definition}

A formula of $\Lwow(L)$ may have countably infinitely many variables, but only finitely many that are free.  Note that the more familiar $\Lww(L)$, consisting of first-order formulas, is the restriction of $\Lwow(L)$ where conjunctions are over \emph{finite} index sets $I$.  As is standard, we will freely use the abbreviations $\forall \defas \neg\exists\neg$ and $\bigvee\defas \neg\bigwedge \neg$, as well as binary $\And$ and $\Or$, in formulas of $\Lwow(L)$.  We will sometimes refer to formulas of $\Lwow(L)$ as \emph{$L$-formulas}.

A sentence is a formula with no free variables.  A \emph{(countable) theory} of $\Lwow(L)$ is an arbitrary (countable) collection of sentences in $\Lwow(L)$. Note that a theory need not be deductively closed.

For a formula $\varphi(x_1, \ldots, x_n)$ of $\Lwow(L)$
whose free variables are among $x_1, \ldots, x_n$, all distinct, the notation $\cM \models \varphi(a_1, \ldots, a_n)$ means that $\varphi(x_1, \ldots, x_n)$ holds in $\cM$ when $x_1, \ldots, x_n$ are instantiated, respectively, by the elements $a_1, \ldots, a_n$ of the underlying set $M$.  
For a theory $T$, we write $\cM\models T$ to mean that $\cM \models \varphi$ for every sentence $\varphi\in T$; in this case, we say that $\cM$ is a \emph{model} of $T$.  We write $T \models \varphi$ to mean that the sentence $\varphi$ is true in every $L$-structure that is a model of $T$.  As a special case, we write $\models \varphi$ to mean $\emptyset \models \varphi$, i.e., the sentence $\varphi$ is true in every $L$-structure.  When $\psi(\xx)$ is a formula with free variables among the entries of the finite tuple $\xx$, we write $ \models \psi(\xx)$ to mean $\models (\forall \xx) \psi(\xx)$.

A key model-theoretic property of $\Lwow(L)$ is that any countable $L$-structure can be characterized up to isomorphism, among countable $L$-structures, by a single sentence of $\Lwow(L)$.

\begin{proposition}[see {\cite[Corollary~VII.6.9]{MR0424560}}
 or {\cite[Theorem~2.4.15]{MR1924282}}]
Let $L$ be a countable language, and let $\cM$ be a countable $L$-structure. 
There is a sentence $\sigma_\cM \in \Lwow(L)$, called
the (canonical) \defn{Scott sentence} of $\cM$, such that for every countable $L$-structure $\cN$, we have $\cN \models \varphi$ if and only if $\cN \cong \cM$.
\end{proposition}


\subsection{Pithy $\Pi_2$ theories}
\label{pithy}
\ \\
\indent
Countable theories consisting of ``extension axioms'' will play a crucial
role in our main construction in Section~\ref{existence}.
  In fact, we will work with a notion of
``one-point extension axioms''\!, which allows us to realize witnesses
for all possible finite configurations, one element at a time.  In a sense
that we make precise in \S\ref{canonical language}, an arbitrary
countable structure is essentially equivalent to one (in a different
language) admitting an
axiomatization
consisting only of one-point extension axioms.  

\begin{definition}
A sentence in $\Lwow(L)$ is $\Pi_2$ when it is of the form $(\forall \xx)(\exists \yy)\psi(\xx, \yy)$, where the (possibly empty) tuple $\xx\yy$ consists of distinct variables, and $\psi(\xx,\yy)$ is quantifier-free.  A countable theory $T$ of $\Lwow(L)$ is $\Pi_2$ when every sentence $\varphi\in T$ is $\Pi_2$.
\end{definition}

In our main construction, it will be convenient to work with a restricted kind
of extension axiom, 
which we call \emph{pithy}.

\begin{definition}
A $\Pi_2$ sentence $(\forall\xx)(\exists\yy)\psi(\xx,\yy) \in \Lwow(L)$, where $\psi(\xx,\yy)$ is quanti\-fier-free, is said to be \defn{pithy} when the tuple $\yy$ consists of precisely one variable.  A countable $\Pi_2$ theory $T$ of $\Lwow(L)$ is said to be pithy when every sentence in $T$ is pithy.  Note that we allow the degenerate case where $\xx$ is the empty tuple and the $\Pi_2$ sentence is of the form $(\exists y)\psi(y)$.
\end{definition}

Note that a pithy $\Pi_2$ sentence can be written uniquely in the form $(\forall\xx)(\exists y)\psi(\xx,y)$, where $\psi$ is quantifier-free, and where the free variables of $\psi$ are among the entries of $\xx y$.


\subsection{The logic action 
on the measurable space $\Model$}\label{logicaction}
\ \\ \indent
Let $L$ be
an arbitrary countable language. Define
$\Model$ to be the set of $L$-structures that have underlying set $\Naturals$.
For every 
formula $\varphi(x_1, \ldots, x_j) \in \Lwow(L)$
and $n_1, \ldots, n_j \in \Naturals$,
where $j$ is the number of free variables of $\varphi$,
define
\[
\llrr{\varphi(n_1, \dots, n_j)} \defas
 \{\M \in \Model \st \M\models \varphi(n_1, \dots, n_j)\}.
\]
The set $\Model$ becomes a measurable space when it is
equipped with the Borel $\sigma$-algebra
generated by 
subbasic open sets of the form:
\[
\llrr{R(n_1, \dots, n_j)}
\]
where $R\in L$ is a $j$-ary relation symbol and $n_1, \ldots, n_j \in \Naturals$;
\[
\llrr{c = n}
\]
where $c\in L$ is a constant symbol and $n \in \Naturals$; and
\[
	\llrr{f(n_1, \dots, n_k) = n_{k+1}}
\]
where $f\in L$ is a $k$-ary function symbol and $n_1, \ldots, n_{k+1}\in \Naturals$.

For any sentence $\varphi$ of $\Lwow(L)$, the set $\llrr{\varphi}$ is Borel, by \cite[Proposition~16.7]{MR1321597}.  Given 
a countable
$L$-structure $\cM$,
recall that the Scott sentence $\sigma_\cM\in\Lwow(L)$ determines $\cM$ up to isomorphism among countable structures. Therefore 
$\llrr{\sigma_\cM} = \{\mathcal{N} \in \Model: \mathcal{N} \cong \cM\}$,
the isomorphism class 
of $\cM$ 
in $\Model$,
is Borel.

Denote by $\sym$ the permutation group of $\Naturals$.  There is a natural group action, called the \emph{logic action}, of $\sym$ on $\Model$, induced by permutation of the underlying set $\Naturals$; for
more details, see \cite[\S16.C]{MR1321597}.  Note that the orbit of an $L$-structure $\cM\in\Model$ under this action is the isomorphism class of $\cM$ in $\Model$.  We call a (Borel) measure $\mu$ on $\Model$ \defn{invariant} when it is invariant under the 
logic action,
i.e., for every Borel set $X \subseteq \Model$ and every $g \in \sym$, we have $\mu(X) = \mu(g\cdot X)$.
Given a subgroup $G$ of $\sym$, 
written $G\le \sym$, a (Borel) measure $\mu$ is 
\emph{$G$-invariant} when it is invariant under the restriction of
the logic action to $G$.

Let $\mu$ be a probability measure on $\Model$.  We say that $\mu$ is
\defn{concentrated} on a Borel set $X \subseteq \Model$ when $\mu(X) = 1$.
We are interested in structures up to isomorphism, and for a countable
infinite $L$-structure $\cM$, we say that \emph{$\mu$ is concentrated on $\cM$} when $\mu$ is concentrated on 
the isomorphism class 
of $\cM$
in $\Model$.  We say that $\cM$ \defn{admits
an invariant measure}
when 
such an 
invariant measure
$\mu$ exists.
Note that when we
	say that $\mu$ is concentrated on some class of structures,
we mean that $\mu$ is concentrated on the restriction of that class to $\Model$.

\subsection{Definable closure}
\label{newstrongamalgamation}
\ \\
\indent
Our characterization of structures admitting invariant measures is in terms
of the group-theoretic notion of \emph{definable closure}.

An \emph{automorphism} of an $L$-structure $\cM$ is a bijection $g\colon M \to M$ such that
\[
R^{\cM}\bigl(g(a_1),\ldots, g(a_j) \bigr) \qquad \text{if and only if} \qquad
R^{\cM}(a_1, \ldots, a_j)
\]
for every relation symbol $R\in L$ of arity $j$ and all elements $a_1, \ldots, a_j \in \cM$,
\[g(c^{\cM}) = c^{\cM}\]
 for every constant symbol $c\in L$,  and 
\[f^{\cM}\bigl(g(b_1),\ldots, g(b_k) \bigr) =
g\bigl(f^{\cM}(b_1,\ldots,b_k )\bigr)  \]
for every function symbol $f\in L$ of arity $k$ and elements $b_1, \ldots,
b_k \in \cM$.

We write $\Aut(\cM)$ to denote the group of automorphisms of $\cM$.

\begin{definition}
Let $\cM$ be an $L$-structure, and let $\aa\in \cM$. The \defn{definable closure} of $\aa$ in $\cM$, denoted $\dcl_\cM(\aa)$, is the collection of $b\in \cM$ that are fixed by all automorphisms of $\cM$ fixing $\aa$ pointwise, i.e., 
the set of $b \in \cM$ for which the set
\[
\bigl\{g(b) \st g \in \Aut(\cM) \text{~s.t.~} (\forall a\in \aa)\ g(a) = a
\bigr\}
\]
is a singleton, namely $\{ b\}$.
\label{dcldef}
\end{definition}

This notion is sometimes known as the \emph{group-theoretic definable closure}.
For countable structures, it has the following equivalent formulation in
terms of the formulas of $\Lwow(L)$ that use parameters from the tuple $\aa$.

Given an $L$-structure $\cM$ and a tuple $\aa\in\cM$, let $L_\aa$ denote the language $L$
expanded by a new constant symbol for each element of $\aa$. Then let
$\cM_\aa$ denote the $L_\aa$-structure 
given by 
$\cM$ with
the 
entries of $\aa$ named by their respective constant symbols in $L_\aa$.

\begin{lemma}[see {\cite[Lemma~4.1.3]{MR1221741}}]
Let $L$ be a countable language, and let $\cM$ be a countable $L$-structure with $\aa\in \cM$. Then $b \in \dcl_\cM(\aa)$ if and only if there is a formula $\varphi\in \Lwow(L_\aa)$ with one free variable,
whose unique realization in $\cM_\aa$ is $b$, i.e.,
\[
\cM_\aa \models \varphi(b) \And [(\forall x, y) (\varphi(x) \And \varphi(y)) \rightarrow x = y].
\]
\label{ctbllemma}
\end{lemma}
\vspace{-15pt}

When the first-order theory of $\cM$  is $\aleph_0$-categorical,
it suffices to consider only first-order formulas $\varphi\in\Lww$ in Lemma~\ref{ctbllemma} (see {\cite[Corollary~7.3.4]{MR1221741}});
in this case, group-theoretic definable closure coincides with
the standard notion of \emph{model-theoretic definable closure}.

\begin{definition}
We say that an $L$-structure $\cM$ has \defn{trivial definable closure}
when the definable closure of every tuple $\aa\in\cM$ is trivial, i.e.,
$\dcl_{\cM}(\aa) = \aa$ for all $\aa \in \cM$.
\end{definition}

Note that if $\cM$ has trivial definable closure, then $L$ cannot have constant symbols, and every function of $\cM$ is a
\emph{choice function} (or \emph{selector}),
i.e., for every function symbol $f\in L$ and every $\aa\in\cM$, we have $f^\cM(\aa) \in \aa$.

It is sometimes more convenient to work with (group-theoretic) \emph{algebraic closure}.

\begin{definition}
Let $\cM$ be an $L$-structure, and let $\aa\in \cM$. The \defn{algebraic closure} of $\aa$
in $\cM$, denoted $\acl_\cM(\aa)$, is the collection of 
$b\in \cM$ whose orbit under those automorphisms of $\cM$ fixing $\aa$ pointwise is finite.
In other words,
$\acl_{\cM}(\aa)$ is the set of $b\in\cM$ for which the set
\[
\bigl\{g(b) \st g \in \Aut(\cM)\text{~s.t.~}(\forall a\in \aa)\ g(a) =
a\bigr\}
\]
is finite.  We say that $\cM$ has \defn{trivial algebraic closure} when
the algebraic closure of every tuple $\aa\in\cM$ is trivial, i.e.,
$\acl_{\cM}(\aa) = \aa$ for all $\aa \in \cM$.
\end{definition}

Note that an $L$-structure has trivial algebraic closure if and only if it has trivial definable closure. This fact will be useful in Section~\ref{examples-nonexamples} when we find examples of structures admitting invariant measures.

\subsection{The canonical language and structure}
\label{canonical language}
\ \\
\indent
We now define the \emph{canonical language} $L_\Abar$ and \emph{canonical structure} $\Abar$ associated to each structure $\cA \in \Model$.
We will see that the canonical structure admits an invariant measure precisely when the original structure does, and has trivial definable closure precisely when the original does. 
In the proof of
our main theorem,
this will
enable us 
to work in 
the setting of canonical structures.  
We will also establish below
that canonical structures admit pithy $\Pi_2$ axiomatizations, 
a fact which we will use in our main construction in Section~\ref{existence}.
For more details on canonical languages and structures, see, e.g., \cite[\S1.5]{MR1425877}.

\begin{definition}
\label{def-canonical}
Let $\cA \in \Model$.
For each
$k\in\Nats$ 
let $\sim_k$ be the
equivalence relation 
on $\Nats^k$ 
given by
\[
\xx \sim_k \yy \qquad \text{if and only if} \qquad (\exists g\in\Aut(\cA))
\ \ 
g(\xx) = \yy.
\]
Define the \defn{canonical language for $\cA$} 
to be the (countable) relational language $L_{\Abar}$ that 
consists of, for
each $k\in\Nats$ and 
$\sim_k$-equivalence class $E$ of $\cA$,
a $k$-ary relation symbol $R_E$.
Then define the \defn{canonical structure for $\cA$} 
to be the structure
$\Abar\in \Str_{L_{\Abar}}$ in which,
for each 
$\sim_k$-equivalence class $E$ of $\cA$,
the interpretation
of $R_E$ is the corresponding orbit $E\subseteq \Nats^k$.
\end{definition}

By the definition of 
$L_\Abar$, 
the $\Aut(\cA)$-orbits of tuples in $\cA$ are $\Lwow(L_\Abar)$-definable in $\Abar$.
In fact, 
as we will see in Lemma~\ref{canon-is-interdefinable},
these orbits are already $\Lwow(L)$-definable in $\cA$.
We begin by noting the folklore result
that the canonical structure $\Abar$ has elimination of quantifiers.

\begin{lemma}
\label{qfdef-canonical}
Let $\cA \in \Model$, and consider its canonical
structure $\Abar$.
Then for all $k\in\Nats$, every $\Aut(\cA)$-invariant subset of
$\Nats^k$ is the 
set
of realizations of some quantifier-free
formula $\psi(\xx) \in \Lwow(L_\Abar)$.
In particular, 
for every formula $\varphi(\xx) \in \Lwow(L_\Abar)$, 
as the set of its realizations is $\Aut(\cA)$-invariant,
there is a quantifier-free
formula $\psi(\xx) \in \Lwow(L_\Abar)$ such that
\[
\Abar \, \models \, \varphi(\xx) \leftrightarrow \psi(\xx)
\]
holds.
\end{lemma}

The following definition of \emph{($\Lwow$-)interdefinability} extends that of the standard
notion of interdefinability from the
setting of $\aleph_0$-categorical theories (see, e.g.,
{\cite[\S1]{MR831437}}).
In particular, two structures are interdefinable when
they have the same underlying set (not necessarily countable) and the same $\Lwow$-definable sets.

Let $L_0$ and $L_1$ be countable languages.
Let $\cN_0$ be an $L_0$-structure and $\cN_1$ an $L_1$-structure having the same underlying set (not necessarily countable).

\begin{definition}
\label{interdefinability}
An \defn{$\Lwow$-interdefinition} (or simply, \emph{interdefinition}) between 
$\cN_0$ and $\cN_1$
is a pair 
$(\Psi_0, \Psi_1)$
of maps
\begin{eqnarray*}
	\Psi_0 \hspace*{-5pt}& : &  \Lwow(L_0) \to \Lwow(L_1) \qquad \text{and}\\
	\Psi_1 \hspace*{-5pt}& : & \Lwow(L_1) \to \Lwow(L_0)
\end{eqnarray*}
satisfying, for $j \in \{0, 1\}$,
\begin{eqnarray*}
\cN_{1-j}
& \models & \ \Psi_{j} \circ \Psi_{1-j} (\eta) \, \leftrightarrow \, \eta,\\
\cN_{1-j}
& \models & \  \neg \Psi_j(\chi) \, \leftrightarrow \, \Psi_j(\neg \chi), \\
\cN_{1-j}
& \models & \  \bigwedge_{i\in I} \Psi_j(\varphi_i) \, \leftrightarrow \,
\Psi_j\bigl(\bigwedge_{i\in I} \varphi_i\bigr), \qquad \text{and} \\
\cN_{1-j}
& \models & \  (\exists x) \Psi_j(\psi(x)) \, \leftrightarrow \,
\Psi_j\bigl((\exists x) \psi(x)\bigr),
\end{eqnarray*}
where 
$\eta\in\Lwow(L_{1-j})$ and
$\chi,\,
\psi(x) 
\in \Lwow(L_j)$,
where $I$ is an arbitrary countable set
and
each $\varphi_i \in \Lwow(L_j)$,
and such that the free variables of $\Psi_j(\xi)$ are the same as those of $\xi$ for every $\xi\in \Lwow(L_j)$.

We say that $\cN_0$ and $\cN_1$ are \defn{interdefinable via $(\Psi_0,
\Psi_1)$} when 
$(\Psi_0, \Psi_1)$  is
an interdefinition 
between 
$\cN_0$ and $\cN_1$
such
that  for every $k \in\Nats$ and every formula $\psi(\xx) \in \Lwow(L_0)$ with 
$k$
free variables, we have
\[
\{ \mm \in \Nats^k \st \cN_0 \models \psi(\mm) \} 
=
\{ \mm \in \Nats^k \st \cN_1 \models \Psi_0(\psi)(\mm) \} 
.
\]
We say that $\cN_0$ and $\cN_1$ are \defn{interdefinable} when they are interdefinable via some interdefinition.
\end{definition}

Note that $\cN_0$ and $\cN_1$ are interdefinable precisely when,
for every $k\in\Nats$,
a
set $X\subseteq \Nats^k$ is definable 
in $\cN_0$
(without parameters)
by an $\Lwow(L_0)$-formula
if and only if 
it is definable 
in $\cN_1$
(without parameters)
by an
$\Lwow(L_1)$-formula.

\begin{lemma}
\label{interdefinition-lemma}
Suppose $(\Psi_0, \Psi_1)$ is an interdefinition between 
$\cN_0$ and $\cN_1$, and let 
$\cN'_0$ 
be an $L_0$-structure 
(not necessarily countable)
that satisfies the same $\Lwow(L_0)$-theory as $\cN_0$.
Then there is a unique $L_1$-structure $\cN'_1$ such that $\cN'_0$ and
$\cN'_1$ are interdefinable via $(\Psi_0, \Psi_1)$. In particular, $\cN'_1$ satisfies
the same $\Lwow(L_1)$-theory as $\cN_1$.
\end{lemma}
\begin{proof}
Let $\cN'_1$ be the unique structure with the same underlying set as $\cN'_0$
such that for
any atomic $L_1$-formula $\psi$, 
the set of realizations of $\psi$ in $\cN'_1$ is precisely the 
set of realizations of $\Psi_1(\psi)$ in $\cN'_0$.
Because $\cN'_0$ 
satisfies the same sentences of $\Lwow(L_0)$ as $\cN_0$, 
by considering 
Definition~\ref{interdefinability}
one can see that
$(\Psi_0, \Psi_1)$
is an interdefinition between $\cN'_0$ and $\cN'_1$;
further, $\cN'_1$ is the only such $L_1$-structure. 
It is immediate that $\cN'_1$ satisfies the same sentences of $\Lwow(L_1)$ as $\cN_1$.
\end{proof}

We will use the following folklore result in the proof of our main theorem.

\begin{lemma}
\label{canon-is-interdefinable}
Let $\cA\in \Model$ and let $\Abar$ be its canonical structure.
Then $\cA$ is interdefinable with $\Abar$.
\end{lemma}

In fact, one can show that 
for $\cA \in \Str_{L_0}$ and $\cB
\in \Str_{L_1}$,
the structures $\cA$ and $\cB$ are interdefinable if and only if $L_\Abar =
L_\Bbar$ and $\Abar = \Bbar$.
Along with Lemma~\ref{canon-is-interdefinable}, this implies that
a structure in $\Str_L$
is characterized up to
interdefinability
by its canonical structure.

As an immediate corollary of Lemmas~\ref{interdefinition-lemma} and \ref{canon-is-interdefinable}, we see that given $\cA\in \Model$ and an arbitrary $L$-structure $\cN$ having the same $\Lwow(L)$-theory as $\cA$, there is a unique $L_\Abar$-structure with which $\cN$ is interdefinable via the interdefinition given in Lemma~\ref{canon-is-interdefinable} between $\cA$ and $\Abar$. 
When $\cN\in\Model$, this $L_\Abar$-structure is $\Nbar$, the canonical structure of $\cN$; in Corollary~\ref{reduct}, we will call the analogous $L_\Abar$-structure $\Nbar$ even when $\cN$ is uncountable.

We now show that 
interdefinability preserves 
whether or not a countable structure admits
an invariant measure and also whether or not it has trivial definable closure.

\begin{lemma}
\label{invmeasinter}
Suppose $\cA\in \Str_{L_0}$ and $\cB\in \Str_{L_1}$ are interdefinable. Then $\cA$ admits an
invariant measure if and only if $\cB$ does.
\end{lemma}
\begin{proof}
Let $(\Psi_0, \Psi_1)$ be an interdefinition between $\cA$ and $\cB$.
We first define a 
Borel 
map
\[
\iota \colon 
\{\cC \in \Str_{L_0} \ \st \ \cC \cong \cA\}
\to
\{\ccD \in \Str_{L_1} \ \st \ \ccD \cong \cB\}
\]
that commutes with the logic action.
For every 
$\cC\in \Str_{L_0}$ isomorphic to $\cA$, 
let $\iota(\cC)$ be
the $L_1$-structure, given by
Lemma~\ref{interdefinition-lemma},
that has the same $\Lwow(L_1)$-theory as $\cB$. 
Since $\iota(\cC)$ and $\cB$ are countable, they
are in fact isomorphic.
Further, since 
$(\Psi_0, \Psi_1)$ is an interdefinition between $\cC$ and $\iota(\cC)$, it follows that
$\iota$ is a bijection.
Recall that the $\sigma$-algebra of $\Str_{L_1}$ is generated by sets of the
form $\llrr{\varphi(n_1, \ldots, n_j)}$, for $\varphi \in \Lwow(L_1)$ and $n_1, \ldots, n_j \in\Nats$, where $j$ is the number of free variables in $\varphi$. 
By the definition of $\iota$, we have
\[
\iota^{-1}\bigl(\llrr{\varphi(n_1, \ldots, n_j)}\bigr) = 
\bigllrr{\Psi_1(\varphi)(n_1, \ldots, n_j)}
,
\]
which is a Borel set in $\Str_{L_0}$.
Hence the map $\iota$ is Borel.

Observe that for every
$g\in\sym$ and 
$\cC\in \Str_{L_0}$ isomorphic to $\cA$,
we have $\iota(g \cdot
\cC) = g\cdot \iota(\cC)$, which is interdefinable with $g \cdot \cC$. 
Hence for every invariant probability measure $\mu$
concentrated on $\cA$, its pushforward along $\iota$ is an invariant
probability measure concentrated on $\cB$. By symmetry, $\cA$ admits an
invariant measure if and only if $\cB$ does.
\end{proof}

In particular, taking $\cB = \Abar$, we see that a 
countable structure admits an invariant measure if and only if 
its canonical structure does.

\begin{lemma}
\label{trivial-dcl-interdefinability}
Suppose $\cA\in \Str_{L_0}$ and $\cB\in \Str_{L_1}$ are interdefinable.
Then $\cA$ has trivial definable closure if and only if $\cB$ does.
\end{lemma}
\begin{proof}
Suppose $\cA$ does not have trivial definable closure. Let
$\aa, b \in \Nats$ with $b\not \in \aa$ and $b\in\dcl(\aa)$.
By Lemma~\ref{ctbllemma}, 
there is a formula $\varphi\in
\Lwow((L_0)_\aa)$
whose unique realization in $\cA_\aa$ is $b$.
Note that $\cA_\aa$ and $\cB_\aa$ are interdefinable.
Hence there is a corresponding
$\Lwow((L_1)_\aa)$-formula whose unique realization in $\cB_\aa$ is $b$,
witnessing the non-trivial definable closure of $\aa$ in $\cB$.
Therefore $\cB$ does not have trivial definable closure either.
The result follows by symmetry.
\end{proof}

By
Lemmas~\ref{invmeasinter} and \ref{trivial-dcl-interdefinability}, 
for countable structures, the properties of having trivial definable closure, and of admitting an invariant measure, are 
determined
up to interdefinability.
Further, by
Lemma~\ref{canon-is-interdefinable}, each of these properties holds of a countable structure if and only if 
the respective property holds
of its canonical structure,
and hence is
determined completely by its automorphism group.

Finally, we show that for every 
countable
structure, there is a pithy $\Pi_2$ theory in its canonical language that
characterizes 
its canonical structure
up to isomorphism among countable structures.
From this, it will follow that the canonical structure is \emph{ultrahomogeneous}.

\begin{definition}
We say that an $L$-structure $\cM$ is \defn{ultrahomogeneous} if any isomorphism between two finitely generated substructures of $\cM$ extends to an automorphism of $\cM$.
\end{definition}

\begin{proposition}
\label{pithypitwotheory-canonical}
Let $\cA\in\Model$. There is a countable 
$\Lwow(L_\Abar)$-theory, every sentence of which is
pithy $\Pi_2$,
and all of whose countable models are
isomorphic to the canonical structure $\Abar$.
\end{proposition}
\begin{proof}
Consider the $\Lwow(L_\Abar)$-theory 
consisting of the following pithy
$\Pi_2$ axioms,
for each $k \in \Nats$:
\begin{itemize}
\item $(\forall \xx)\, \bigl(R_E(\xx) \leftrightarrow \bigwedge
\{ \neg R_{G}(\xx)
\st R_{G} \neq R_{E} \text{~is a $k$-ary relation symbol in $L_\Abar$}\}\bigr)$, 
\item $(\forall \xx)\, \bigvee\{
R_{G}(\xx) \st R_{G} \text{~is a $k$-ary relation symbol in $L_\Abar$}\}$, and
\item $(\forall \xx)\, \bigl( R_{E}(\xx) \, \to \, (\exists y)\, R_{F}(\xx, y)
\bigr)
$,
\end{itemize}
where $|\xx| = k$, and $R_{E}$ and $R_{F}$ are, respectively, $k$- and $(k+1)$-ary relation
symbols in $L_\Abar$ such that
\[
\Abar\, \models \,(\forall \xx y) \, \bigl (R_F(\xx, y) \to R_E(\xx) \bigr)
. 
\]

It is immediate that $\Abar$ satisfies this theory. Furthermore, for any
two countable models of the theory, the first two axioms require that
every $k$-tuple in either model
realizes exactly one $k$-ary relation.
Hence given two $k$-tuples $\aa, \bb$ of $\Abar$ satisfying the same relation, we may use
the third axiom 
to
construct an automorphism of $\Abar$ mapping $\aa$ to $\bb$,
by a standard back-and-forth argument.
This establishes that the theory has one countable model up to isomorphism.
\end{proof}

Note that the above argument further shows the standard result that $\Abar$ is ultrahomogeneous.
The pithy $\Pi_2$ theory of Proposition~\ref{pithypitwotheory-canonical} can therefore be thought of as an infinitary
analogue of a \Fr\ theory.
In particular, as with \Fr\ theories in
first-order relational languages, the age of 
$\Abar$ has strong amalgamation
precisely when 
$\Abar$ has trivial definable closure.
(For more details on \Fr\ theories, see \cite[\S7.1]{MR1221741}.)
Therefore, even if $\cA$ is not ultrahomogeneous itself, 
Corollary~\ref{maincorollary} could be applied
to a
structure
that is essentially equivalent to $\cA$,
namely the canonical structure $\Abar$. Indeed,
by Lemma~\ref{trivial-dcl-interdefinability},
$\Abar$ has
strong amalgamation precisely when $\cA$ has trivial definable closure.


\subsection{Basic probability notions}
\ \\
\indent
Throughout this paper, we make use of conventions from measure-theoretic probability theory to talk about random structures having certain almost-sure properties. For a general reference on probability theory, see, e.g., \cite{MR1876169}.

Let $(\Ohm,\mathcal{G}, \Pr)$ be a probability space, and suppose $(H, \mathcal{H})$ is a measurable space.  Recall that an \emph{$H$-valued random variable} $Z$ is a $(\mathcal{G},\mathcal{H})$-measurable function \linebreak $Z\colon\Ohm\to H$.  Such a function $Z$ is also sometimes called a \emph{random element in} $H$.  The \emph{distribution} of $Z$ is defined to be the probability measure $\Pr\circ Z^{-1}$.

Given a property $E\in \mathcal{H}$, we say that $E$ holds of $Z$ \emph{almost surely}, abbreviated \emph{a.s.}, when $\Pr\bigl(Z^{-1}(E)\bigr) = 1$.  Sometimes, in this situation, we say instead that $E$ holds of $Z$ \emph{with probability one}.  For example, given a random element $Z$ in $\Model$ and a Borel set $\llrr{\varphi}$, where $\varphi$ is a sentence of $\Lwow(L)$, we say that $\llrr{\varphi}$ holds of $Z$ a.s.\ when $\Pr\bigl(Z^{-1}(\llrr{\varphi})\bigr) = \Pr\bigl(\{ w \in \Ohm \st Z(w) \models \varphi\}\bigr) = 1$.  In fact, we will typically not make the property explicit, and will, for instance, write that the random structure $Z \models \varphi$~a.s.\ when $\Pr\bigl(\{ w \in \Ohm \st Z(w) \models \varphi\}\bigr) = 1$; 
this probability is abbreviated as  \linebreak $\Pr\{ Z \models \varphi\}$.

In the proof of our main theorem, when we show that a measure $\mu$ on $\Model$ is concentrated on the set of models in $\Model$ of some sentence $\varphi$, we will do so by demonstrating that, with probability one, $Z \models \varphi$, where $Z$ is a random structure with distribution $\mu$.

A sequence of ($H$-valued) random variables is said to be
\emph{independent and identically distributed}, abbreviated \emph{i.i.d.},
when each random variable has the same distribution and the random
variables are mutually independent. When this distribution is $m$, we say that the sequence
is \emph{$m$-i.i.d.}


\section{Existence of invariant measures}
\label{existence}

We now 
show the
existence
of invariant measures concentrated on a countable infinite structure having trivial definable closure.  
We prove this in Theorem~\ref{InvariantMeasuresConcentrated}, which constitutes one direction of
Theorem~\ref{maintheorem},
the main result of this paper.

The following is an outline of our proof; 
in the presentation below we
will, however, develop the machinery in the reverse order.  Let $\cM\in\Model$ be a
countable infinite $L$-structure 
having
trivial definable closure, and $L_\Mbar$ its canonical language and $\Mbar$ its canonical structure,
as in \S\ref{canonical language}.
Let
$T_\Mbar$ be a
countable pithy $\Pi_2$ 
theory 
of $\Lwow(L_\Mbar)$
all of whose countable models are isomorphic  to $\Mbar$,
as in 
Proposition~\ref{pithypitwotheory-canonical}.
We show, in \S\ref{invariant}, that such a theory
$T_\Mbar$
has a
property that we call \emph{duplication of quantifier-free types}.  In
\S\S\ref{duplication}--\ref{construction}, we use this property to build a
\emph{Borel $L_\Mbar$-structure} $\PP$ that \emph{strongly witnesses
$T_\Mbar$}.
Roughly speaking, this means that $\PP$ is an $L_\Mbar$-structure
with underlying set
$\Reals$, whose relations
are
Borel,
such that for every pithy $\Pi_2$ sentence 
$(\forall \xx )(\exists y) \varphi(\xx,y)\in T_\Mbar$
and tuple $\aa\in\PP$ of the appropriate length,
either there is a ``large'' set of elements $b\in\PP$ such that
$\varphi(\aa, b)$ holds, or else there is some $b\in\aa$ such that
$\varphi(\aa, b)$ holds; in either case, $(\exists y) \varphi(\aa,y)$ is
``witnessed''.  In \S\S\ref{BorelLStructuresSection}--\ref{WitnessesT} we
show how to use a Borel $L_\Mbar$-structure that strongly witnesses
$T_\Mbar$
to produce
an invariant measure concentrated on the set of models of 
$T_\Mbar$
that are in
$\Str_{L_\Mbar}$.
By the initial choice of 
$T_\Mbar$,
this invariant measure 
on $\Str_{L_\Mbar}$ is concentrated on $\Mbar$.
By results of \S\ref{canonical language}, we obtain
an invariant measure on $\Model$ concentrated on $\cM$.

\subsection{Sampling from Borel $L$-structures}
\label{BorelLStructuresSection}
\ \\ \indent
We begin by introducing a certain kind of $L$-structure with underlying set $\Reals$, whose relations and functions
 are Borel (with respect to the standard topology on $\Reals$).  
Our definition is motivated by Petrov and Vershik's notion of a \emph{Borel graph} \cite[Definition~1]{MR2724668}.
The model theory of such Borel structures has earlier been studied by Harvey
Friedman (published in \cite{MR819547}). For a survey, including more
recent work, see \cite[\S1]{MR3205056}.

\begin{definition}
\label{BorelLstructure}
Let $\PP$ be an $L$-structure whose underlying set is $\Reals$.  We say
that $\PP$ is a \defn{Borel $L$-structure} if for all relation symbols
$R \in L$, the set $\bigl\{\aa \in \PP^j \st R^\PP(\aa)\bigr\}$
is a Borel subset of $\Reals^j$, where $j$ is the arity of $R$; and for all function symbols $f\in L$, the function $f^\PP\colon \PP^k \to \PP$ is Borel (equivalently, the graph of $f^\PP$ is Borel),
where $k$ is the arity of $f$.
\end{definition}

Note that although structures with underlying set $\Reals$ will suffice for our
purposes, we could have defined the notion of a \emph{Borel $L$-structure} 
more generally, for other measure spaces.

Our first goal in \S\ref{BorelLStructuresSection} is to define a sampling procedure that, given a Borel $L$-structure with certain properties,
yields an invariant measure on $\Str_L$. We begin with several definitions.

Given an $L$-structure $\cN$ of arbitrary cardinality, we write $\Sub(\cN)$ to denote the set of those sequences in $\cN^\omega$ that contain all constants of $\cN$ and are closed under the application of functions of $\cN$. Such a sequence is precisely an enumeration (possibly with repetition) of the underlying set of some countable substructure of $\cN$.
Note that whenever $L$ is relational, or when $\cN$ has trivial definable closure, we have
$\Sub(\cN) = \cN^\omega$ (but not conversely).
We say that $\cN$ is \defn{samplable} when 
$\Sub(\cN) = \cN^\omega$.
Observe that $\cN$ is samplable precisely 
when $L$ has no constant symbols and every function is a choice function.

Next we describe a map taking an element of $\Sub(\cN)$ to an
$L$-structure with underlying set $\Naturals$. 
In the case when  $\cN$ is samplable, we will apply
this map to a random sequence of elements of $\cN$
to induce a random $L$-structure with underlying set $\Naturals$.

\begin{definition}
Suppose $\cN$ is 
an 
$L$-structure (of arbitrary cardinality).
Define the function $\F_{\cN}\colon 
\Sub(\cN)
\to \Model$  as follows.
For 
$\AA=(a_i)_{i\in \w} \in \Sub(\cN)$,
let $\F_{\cN}(\AA)$ be the $L$-structure with underlying set $\Naturals$ satisfying
\[
\F_{\cN}(\AA) \, \models \,R(n_1, \dots, n_j) 
\qquad \text{if and only if} \qquad
\cN\models R(a_{n_1}, \dots, a_{n_j})
\]
for every relation symbol $R \in L$ and 
for all $n_1, \dots, n_j \in \Naturals$, where $j$ is the arity of $R$;
satisfying
\[
\F_{\cN}(\AA) \, \models \,
(c = n)
\qquad \text{if and only if} \qquad
\cN\models 
(c = a_n)
\]
for every constant symbol $c \in L$ and for all $n\in\Nats$; 
satisfying
\[
\F_{\cN}(\AA) \, \models \,
f(n_1, \ldots, n_k)  = n_{k+1}
\qquad \text{if and only if} \qquad
\cN\models f(a_{n_1}, \ldots, a_{n_k})  = a_{n_{k+1}}
\]
for every function symbol $f \in L$ and
for all $n_1, \dots, n_{k+1} \in \Naturals$, where $k$ is the arity of $f$;
and for which equality is inherited from $\Naturals$, i.e.,
\[
\F_{\cN}(\AA)
\, \models \,(m \neq  n)
\]
just when $m$ and $n$ are distinct natural numbers.
\end{definition}
When the sequence $\AA\in\Sub(\cN)$ has no repeated entries,
$\F_{\cN}(\AA) \in\Model$ is iso\-morphic to a countable infinite substructure of $\cN$.
In fact, this will hold a.s.\ for the random $L$-structures that we construct in
\S\ref{WitnessesT}.

Recall the definition of the Borel $\sigma$-algebra on $\Model$ in \S\ref{logicaction}.
Define a \defn{subbasic formula} of $\Lwow(L)$ to be a formula
of the form $x_1 = x_2$, $R(x_1, \ldots, x_j)$, 
$c = x_1$, or $f(x_1, \ldots, x_k) = x_{k+1}$, 
where $R\in L$ is a relation symbol and $j$ its arity, 
$c\in L$ is a constant symbol, $f\in L$ is a function symbol and $k$ its arity, and the 
$x_i$
are distinct variables.

\begin{lemma}
Let $\PP$ be a Borel $L$-structure.
Then $\F_{\PP}$ is a Borel measurable function.
\end{lemma}
\begin{proof}
It suffices to show that the preimages of subbasic open sets of $\Model$ are Borel.
Let $\zeta$ be a subbasic formula of $\Lwow(L)$ with $j$ free variables, and let $n_1, \ldots, n_j \in \Naturals$.
We wish to
show that $\F_{\PP}^{-1}\Bigl(\llrr{\zeta(n_1, \dots, n_j)}\Bigr)$ is Borel.

Let $\pi_{n_1, \dots, n_j} \st \PP^\w \rightarrow \PP^j$ be the projection
map defined by
\[
\pi_{n_1, \dots, n_j}\bigl((a_i)_{i\in \w}\bigr) =
(a_{n_1}, \dots, a_{n_j});
\] this map is Borel. Then
\[
\F_{\PP}^{-1} \Bigl(\llrr{\zeta(n_1, \dots, n_j)}\Bigr)
= \pi_{n_1, \dots, n_j}^{-1}\bigl(\{\aa \in \PP^j\st \PP\models \zeta(\aa)\}\bigr),
\]
as both sides of the equation are equal to
\[ \bigl\{(a_i)_{i\in\omega}\in\PP^\omega \st \PP \models \zeta(a_{n_1}, \ldots,
a_{n_j})\bigr\}.
\]
By Definition \ref{BorelLstructure} we have that $\{\aa \in \PP^j \st
\PP\models \zeta(\aa)\}$ is Borel. 
Hence \linebreak $\F_{\PP}^{-1}\Bigl(\llrr{\zeta(n_1, \dots, n_j)}\Bigr)$ is also Borel, as desired.
\end{proof}

We now
show how to induce an invariant measure on $\Model$ from a samplable Borel $L$-structure $\PP$.
Suppose $m$ is a 
probability measure on $\Reals$. Denote by $m^\infty$ the corresponding product measure on $\Reals^{\w}$, i.e., the distribution
of a sequence of independent samples from
$m$.
Note that $m^\infty$ is invariant under arbitrary reordering of the indices.
We will obtain an invariant measure on $\Model$ by taking the distribution of the \emph{random} structure with underlying set $\Naturals$ corresponding to an $m$-i.i.d.\ sequence of elements of 
$\PP$.

This technique for constructing invariant measures by sampling a continuum-sized structure was used by
Petrov and Vershik
\cite{MR2724668},
and the following notation and results parallel those in
\cite[\S2.3]{MR2724668}.  A similar method of sampling is used 
in \cite[\S2.6]{MR2274085}
to
produce the countable random graphs known as \emph{W-random graphs} from
continuum-sized \emph{graphons}; 
for more details on the relationship between 
these notions and our
construction,
see
\S\ref{graphlimits}.

\begin{definition}
\label{mupm}
Let $\PP$ be a Borel $L$-structure,
and let $m$ be a probability measure on $\Reals$.
Define the measure $\mu_{(\PP,m)}$ on $\Model$ to be
\[
\mu_{(\PP, m)} \defas
 m^{\infty}\circ \F_{\PP}^{-1}.
\]

When 
$\PP$ is samplable,
$m^\infty\bigl(\F_\PP^{-1}(\Model)\bigr) =1$, and so
$\mu_{(\PP,m)}$ is a probability measure, namely the distribution of a random
element in $\Model$
induced via $\F_{\PP}$ by an $m$-i.i.d.\ sequence on $\Reals$.
\end{definition}

The following lemma makes precise the sense in which the invariance of $m^\infty$
(under the action of $\sym$ on $\Reals^\omega$) yields the invariance of $\mu_{(\PP,m)}$ (under the logic action).

\begin{lemma}
\label{mu-is-invariant}
Let $\PP$ be a Borel $L$-structure,
and let $m$ be a probability measure on $\Reals$.
Then the measure $\mu_{(\PP, m)}$ is invariant under the logic action.
\end{lemma}
\begin{proof}
It suffices to verify that $\mu_{(\PP,m)}$ is invariant on a $\pi$-system
(i.e., a family of sets closed under finite intersections) that generates
the Borel $\sigma$-algebra on $\Model$, by \cite[Lemma~1.6.b]{MR1155402}. 
We first show
that $\mu_{(\PP,m)}$ is invariant 
on subbasic open sets
determined by subbasic formulas 
of $\Lwow(L)$
along with tuples instantiating them.
We then
show its invariance for the $\pi$-system consisting of sets 
determined by finite conjunctions 
of such subbasic formulas.

Let $\zeta$ be a subbasic formula of $\Lwow(L)$ with $j$ free variables,
and let $n_1, \ldots,
n_j\in\Naturals$.
Consider the set 
$\llrr{\zeta(n_1, \ldots, n_j)}$,
and let $g\in \sym$.
Note that
\[
\bigllrr{\zeta \bigl( g(n_1), \dots, g(n_j)\bigr)}
=
\bigl \{ g \cdot \XX \st \XX \in \llrr{\zeta(n_1, \dots, n_j)} \bigr \},
\]
where $\,\cdot\,$ denotes the logic action of $\sym$ on $\Model$.
We will show that
\begin{equation}
\label{goal}
\mu_{(\PP,m)}\Bigl(\bigllrr{\zeta \bigl(g(n_1), \dots, g(n_j)\bigr)} \Bigr)
=
\mu_{(\PP,m)}\Bigl(\llrr{\zeta(n_1, \dots, n_j)}\Bigr).\tag{$\star$}
\end{equation}
We have
\[
\F_{\PP}^{-1} \Bigl(
\bigllrr{\zeta \bigl( g(n_1), \dots, g(n_j)\bigr)}
  \Bigr)
=
\pi_{g(n_1), \dots, g(n_j)}^{-1}\bigl(\{\aa \in \PP^j\st \PP\models \zeta(\aa)\}\bigr)
\]
and
\[
\F_{\PP}^{-1} \Bigl(
\llrr{\zeta ( n_1, \dots, n_j )}
  \Bigr)
=
\pi_{n_1, \dots, n_j}^{-1}\bigl(\{\aa \in \PP^j\st \PP\models \zeta(\aa)\}\bigr).
\]
Because $m^\infty$  is invariant under the action of $\sym$ on
$\Reals^\omega$ (given by permuting coordinates of $\Reals^\omega$), the 
Borel subsets 
\[
\F_{\PP}^{-1} \Bigl(
\bigllrr{\zeta \bigl( g(n_1), \dots, g(n_j)\bigr)}
  \Bigr)
\]
and
\[
\F_{\PP}^{-1} \Bigl(
\llrr{\zeta ( n_1, \dots, n_j)}
  \Bigr)
\]
of $\Reals^\omega$ have equal $m^\infty$-measure, and so \eqref{goal} holds.

Now consider 
subbasic formulas
$\zeta_1, \ldots, \zeta_k$  of $\Lwow(L)$ with $j_1, \ldots, j_k$ free variables, respectively,
and let $j' \defas \sum_{i=1}^k j_k$. Denote the conjunction of 
these formulas on
non-overlapping variables by 
\[\varphi (x_1, \ldots,
x_{j'})\defas  \zeta_1(x_1, \ldots, x_{j_1}) 
\And \cdots \And \zeta_k(x_{j'-j_k+1}, \ldots, x_{j'}),\]
where $x_1, \ldots, x_{j'}$ are distinct variables.
We have 
\[
\llrr{\varphi(n_1, \ldots, n_{j'})} = 
\bigllrr{\zeta_1(n_1, \ldots, n_{j_1})}
\cap \cdots \cap
\bigllrr{\zeta_k(n_{j'-j_k+1}, \ldots, n_{j'})},
\]
for all $n_1, \ldots, n_{j'} \in \Naturals$ (not necessarily distinct),
and from this we see that
\[
	\bigllrr{\varphi\bigl(g(n_1), \ldots, g(n_{j'})\bigr)} = 
	\bigllrr{\zeta_1\bigl(g(n_1), \ldots, g(n_{j_1})\bigr)}
\cap \cdots \cap
\bigllrr{\zeta_k\bigl(g(n_{j'-j_k+1}), \ldots, g(n_{j'})\bigr)}.
\]
	Hence
\[
\F_{\PP}^{-1} \Bigl(
\bigllrr{\varphi \bigl( g(n_1), \dots, g(n_{j'})\bigr)}
  \Bigr)
\]
and
\[
\F_{\PP}^{-1} \Bigl(
\llrr{\varphi ( n_1, \dots, n_{j'})}
  \Bigr)
\]
have equal measure under $m^\infty$.
Hence $\mu_{(\PP,m)}$  is invariant under the logic action.
\end{proof}


Recall that a measure on $\Reals$ is said to be \defn{continuous} (or
\emph{nonatomic}) if it assigns measure zero to every singleton.
When $m$ is continuous, samples from $\mu_{(\PP, m)}$ are a.s.\ isomorphic to substructures of $\PP$.

\begin{lemma}
\label{continuouslemma}
Let $\PP$ be a samplable Borel $L$-structure
and let $m$ be a continuous probability measure on $\Reals$.  Then 
$\mu_{(\PP, m)}$ 
is 
a probability measure on $\Model$ that is
concentrated on the 
union of
isomorphism classes 
of countable infinite substructures of $\PP$.
\end{lemma}
\begin{proof}
As noted before, because $\PP$ is samplable, 
$m^\infty\bigl(\F_\PP^{-1}(\Model)\bigr) =1$, and so
$\mu_{(\PP,m)}$ is a probability measure.
Let
$\AA = (a_i)_{i\in \w}$ be
an $m$-i.i.d.\ sequence of $\Reals$.
Note that the induced countable structure $\F_{\PP}(\AA)$ is now a \emph{random $L$-structure}, i.e., a $\Model$-valued random variable, whose distribution is $\mu_{(\PP,m)}$.
Because $m$ is continuous, and since for any $k\ne \ell$ the random variables $a_k$
and $a_{\ell}$ are independent,  the sequence $\AA$ has no repeated entries
a.s.
Hence
$\F_{\PP}(\AA)$
is a.s.\ isomorphic to
a countable infinite (induced) substructure of $\PP$.
\end{proof}

\subsection{Strongly witnessing a pithy $\Pi_2$ theory}
\label{WitnessesT}
\ \\ \indent
We have seen how to construct invariant measures on $\Model$ by sampling
from a samplable Borel $L$-structure.  
We now describe
a property 
that will give us sufficient control over 
such measures 
to 
ensure that they are concentrated on the set of models, in $\Model$, of a given
countable pithy $\Pi_2$ theory $T$ of $\Lwow(L)$.
For this we define when
a Borel $L$-structure $\PP$ and a measure $m$ \emph{witness} $T$,
generalizing the key property of Petrov and Vershik's \emph{universal
measurable graphs} in \cite[Theorem~2]{MR2724668}.  From this, we define
when $\PP$ \emph{strongly witnesses} $T$, a notion that we find more
convenient to apply.

We begin by defining the notions of \emph{internal} and \emph{external witnesses}.
\begin{definition}
Let $\cM$ be an $L$-structure containing a tuple $\aa$, and let
$\psi(\xx,y)$ be
a quantifier-free formula of $\Lwow(L)$,
all of whose free variables are among $\xx y$.
We say that an element $b\in\cM$ is a
\defn{witness} for 
$(\exists y)\psi(\aa, y)$ when $\cM\models\psi(\aa, b)$.  We say that such an element $b$ is an \defn{internal witness} when $b\in\aa$, and an \defn{external witness} otherwise.
\end{definition}

Recall that a measure $m$ on $\Reals$ is said to be \defn{nondegenerate} when every nonempty
open set has positive measure.

\begin{definition}
\label{Twitness}
Let $\PP$ be a Borel $L$-structure and let $m$ be a probability measure on $\Reals$.  Suppose $T$ is a countable pithy $\Pi_2$ theory of $\Lwow(L)$.  We say that the pair $(\PP, m)$ \defn{witnesses $T$} if for every sentence $(\forall \xx)(\exists y)\psi(\xx,y) \in T$, and for every tuple $\aa \in \PP$ such that $|\aa| = |\xx|$,  we have either
\begin{itemize}
\item[\emph{(i)}]
\quad $\PP \models \psi(\aa,b)$ for some $b\in\aa$, or
\vspace{5pt}
\item[\emph{(ii)}]
\quad $m\bigl(\{ b \in \PP \st \PP \models \psi(\aa,b)\}\bigr ) > 0$.
\end{itemize}
We say that \defn{$\PP$ strongly witnesses $T$} when, for every
nondegenerate probability measure $m$ on $\Reals$, the pair $(\PP,m)$ witnesses $T$.
\end{definition}

Intuitively, the two possibilities \emph{(i)} and \emph{(ii)} say that
witnesses for $(\exists y) \psi(\aa,y)$ are easy to find: Either an internal witness already exists among the parameters $\aa$, or else 
witnesses are plentiful elsewhere in the structure $\PP$, according to $m$.

Strong witnesses
simply 
allow
us to work without keeping
track of a measure $m$.  In fact, when we build structures $\PP$ that
strongly witness a theory, in \S\ref{construction}, we will be 
more concrete, by declaring entire intervals to be external witnesses.

These definitions generalize two of the key notions in \cite{MR2724668}.  
Let $L_G$ be the language of graphs.
A \emph{universal measurable graph} $(X,m,E)$  as defined in
\cite[Definition~3]{MR2724668} roughly corresponds to a Borel
$L_G$-structure $(X,E)$ with vertex set $X$ and edge relation $E$ for which
$((X,E), m)$ witnesses the theory of the Rado graph $\Rado$.
A \emph{topologically universal graph} $(X,E)$ as defined in
\cite[Definition~4]{MR2724668} roughly corresponds to a Borel 
$L_G$-structure that strongly witnesses the theory of 
$\Rado$
by virtue of entire intervals being witnesses.  
Theorem~\ref{BorelLStructuresWitnessingTLeadToInvariantMeasures}
and Corollary~\ref{stronglywitnessingcorollary} below
are inspired directly by Petrov and Vershik's constructions.

We will use (continuum-sized) samplable Borel $L$-structures strongly witnessing a
countable pithy $\Pi_2$ theory $T$ to produce random \emph{countable} structures that satisfy $T$
almost surely.
However, we first note that 
the property of strongly witnessing such a $T$ is powerful enough to ensure that 
a samplable
Borel $L$-structure is itself a model of $T$.

\begin{lemma}
\label{SubmodelsSatisfyingT}
Let $\PP$ be a samplable Borel $L$-structure, and let $T$ be a countable pithy $\Pi_2$ theory of $\Lwow(L)$.  If $\PP$ strongly witnesses $T$, then $\PP\models T$.
\end{lemma}
\begin{proof}
Fix a pithy $\Pi_2$ sentence $(\forall \xx)(\exists y)\psi(\xx,y) \in T$.
Suppose $\aa \in \PP$, where $|\aa| = |\xx|$. If possibility \emph{(i)} of
Definition~\ref{Twitness} holds, then there is an internal witness for $(\exists y)\psi(\aa,y)$, i.e., there is some $b\in\aa$ such that $\PP\models \psi(\aa, b)$.  Otherwise, possibility \emph{(ii)} holds, and so the set $\{ b \in \PP \st \PP \models \psi(\aa,b)\}$ of external witnesses has positive $m$-measure for an arbitrary nondegenerate probability measure $m$ on $\Reals$; in particular,
this set is nonempty.  Either way, for all $\aa\in\PP$ we have $\PP\models (\exists y)\psi(\aa, y)$, and therefore $\PP \models (\forall \xx)(\exists y)\psi(\xx, y)$. Thus $\PP \models T$.
\end{proof}


When the measure $m$ is continuous, samples from $\mu_{(\PP,m)}$ are a.s.\ models of $T$.

\begin{theorem}
\label{BorelLStructuresWitnessingTLeadToInvariantMeasures}
Let $T$ be a countable pithy $\Pi_2$ theory of $\Lwow(L)$, and let $\PP$ be a samplable Borel $L$-structure.
Suppose $m$ is a continuous probability measure on $\Reals$ such that $(\PP,m)$ witnesses $T$.  Then $\mu_{(\PP,m)}$ is concentrated on the set of structures in $\Model$ that are models of $T$.
\end{theorem}
\begin{proof}

Let $\AA = (a_i)_{i \in \w}$ be an $m$-i.i.d.\ sequence of elements of
$\PP$.  Recall that by the proof of Lemma~\ref{continuouslemma},
$\mu_{(\PP,m)}$ is the distribution of the random
structure $\F_{\PP}(\AA)$, and so we must show that $\F_{\PP}(\AA) \models T$ a.s.  Because $T$ is countable, it suffices by countable additivity to show that for any sentence $\varphi \in T$, we have $\F_{\PP}(\AA) \models \varphi$ a.s.

Recall that $T$ is a countable pithy $\Pi_2$ theory.  Suppose $(\forall \xx)(\exists y)\psi(\xx,y) \in T$, and let $k = |\xx|$ (which may be $0$).  Our task is to show that, with probability one,
\[
\F_{\PP}(\AA) \models (\forall \xx)(\exists y)\psi(\xx,y)  .
\]
Fix $t_1\cdots t_k\in \Naturals$.  We will show that, with probability one,
\begin{equation}
\F_{\PP}(\AA) \models (\exists y) \psi(t_1\cdots t_k, y).\tag{$\dagger$}
\label{mainclaim}
\end{equation}
Consider the random tuple $a_{t_1}\cdots a_{t_k}$. Because $\PP$ strongly witnesses $T$, by Definition~\ref{Twitness} it is \emph{surely} the case that either
\begin{itemize}
\item[\emph{(i)}]
for some $\ell$ such that $1\le \ell \le k$, 
the random real $a_{t_\ell}$ is an internal witness for $(\exists y)
\psi(a_{t_1} \cdots a_{t_k}, y)$, i.e.,
\[
\PP \models \psi(a_{t_1} \cdots a_{t_k}, a_{t_\ell}),
\]
\end{itemize}
\vspace*{-3pt}
or else 
\begin{itemize}
\item[\emph{(ii)}] the (random) set 
of 
witnesses for $(\exists y) \psi(a_{t_1} \cdots a_{t_k}, y)$
has positive measure, i.e.,
\[
m\bigl( \{ b\in\Reals \st    \PP \models \psi(a_{t_1} \cdots a_{t_k}, b) \} \bigr) > 0 .
\]
\end{itemize}

In case $\emph{(i)}$, we have
\[
\F_{\PP}(\AA) \models \psi(t_1 \cdots t_k, t_\ell),
\]
where $\ell$ is as above, and so \eqref{mainclaim} holds surely.

Now suppose case \emph{(ii)} holds, and condition on 
$a_{t_1} \cdots a_{t_k}$. 
Then
\[
\beta \defas
m\bigl( \{ b\in\Reals \st    \PP \models \psi(a_{t_1} \cdots a_{t_k}, b)
\} \bigr)
\]
is a positive constant.
For each $n\in\Naturals$ 
not among $t_1, \ldots, t_k$, the random element $a_n$ is $m$-distributed,
and so  the 
events
\[
\PP \models \psi(a_{t_1} \cdots a_{t_k}, a_n)
\]
each have probability $\beta$.
These events are also mutually independent for such $n$,
and so
with probability one, there is some $s \in \Naturals$ for which
\[
\F_{\PP}(\AA) \models \psi(t_1 \cdots t_k, s).
\]
Therefore, in this case, \eqref{mainclaim} holds almost surely.
\end{proof}

Finally, we show that given a Borel $L$-structure strongly witnessing $T$, we can construct an invariant measure concentrated on the set of models of $T$ in $\Model$.

\begin{corollary}
\label{stronglywitnessingcorollary}
Let $T$ be a countable pithy $\Pi_2$ theory of $\Lwow(L)$, and let $\PP$ be a samplable Borel $L$-structure.
Suppose that $\PP$
strongly witnesses $T$.  Then there is an invariant measure on $\Model$ that is concentrated on the set of structures in $\Model$ that are models of $T$.
\end{corollary}
\begin{proof}
Let $m$ be a nondegenerate probability measure on $\Reals$ that is continuous (e.g., a Gaussian or Cauchy distribution).  
By Lemmas~\ref{mu-is-invariant}
and \ref{continuouslemma},
$\mu_{(\PP,m)}$ 
is an
invariant measure,
and
by Theorem~\ref{BorelLStructuresWitnessingTLeadToInvariantMeasures}, 
it is 
concentrated on 
the set of models of $T$ in $\Model$.
\end{proof}

Our constructions of invariant measures all employ 
a samplable Borel $L$-structure
$\PP$ that strongly witnesses a countable pithy $\Pi_2$ theory $T$.  We note that the machinery developed in \S\ref{WitnessesT} could have been used to build invariant measures via the substantially weaker condition that, for a given probability measure $m$, for
each 
expression $(\exists y) \psi(\aa, y)$ for which $(\forall \xx)(\exists y) \psi(\xx, y)\in T$,
there are witnesses in $\PP$
only for $m$-almost all tuples $\aa$.  However, in this case, $\PP$ need not be a model of $T$ (in contrast to Lemma~\ref{SubmodelsSatisfyingT}), nor need $(\PP,m)$ even witness $T$.  Also, while we defined the notion of witnessing only for $L$-structures on $\Reals$, we could have developed a similar 
notion
for $L$-structures whose underlying set is an $m$-measure one subset of $\Reals$.

Next we
find conditions that allow us to construct samplable Borel $L$-structures 
that strongly witness $T$.  When all countable models of $T$ 
are isomorphic to
a particular countable infinite $L$-structure $\cM$, we
will thereby obtain an invariant measure concentrated on the isomorphism class of $\cM$ in $\Model$.

\subsection{Duplication of quantifier-free types}\label{duplication}
\ \\ \indent
We now introduce the notion of a theory having \emph{duplication of quantifier-free types}.  We will see in \S\ref{construction} that when $L$ is a countable relational language and $T$ is a countable pithy $\Pi_2$ theory of $\Lwow(L)$, duplication of quantifier-free types guarantees the existence of a samplable Borel $L$-structure strongly witnessing $T$.  However, for the definitions and results in \S\ref{duplication}, we do not require $L$ to be relational.

We first recall the notion of a \emph{quantifier-free type}, which can be thought of as giving the full description of 
the subbasic formulas
that 
could 
hold among the elements 
of some tuple.  In first-order logic this is typically achieved with an infinite consistent set of formulas, but in our infinitary context, a single satisfiable formula of $\Lwow(L)$ suffices.

\begin{definition}
\label{newTypeDef}
Suppose $\xx$ is a finite tuple of variables.  Define a \defn{(complete) quantifier-free type $p(\xx)$ of $\Lwow(L)$} to be a quantifier-free formula in $\Lwow(L)$, whose free variables are precisely those in $\xx$, such that 
the sentence $(\exists \xx)p(\xx)$ has a model, and such that  for
every quantifier-free formula $\varphi(\xx) \in \Lwow(L)$, either
\[
\models  p(\xx) \rightarrow \varphi(\xx)
\text{~~\ \ \ \ ~~~~or~~~\ \ \ \ ~~~}
\models  p(\xx) \rightarrow \neg \varphi(\xx)
\]
holds.
\end{definition}

Note that because
we require this condition only for
quantifier-free
formulas $\varphi$, it suffices for $p(\xx)$ to be a quantifier-free formula such that whenever $\zeta(\xx)$ is a subbasic formula,
and whenever $\yy$ is a tuple of length equal to the number of free variables of $\zeta$ such that every variable of $\yy$ is in the tuple $\xx$, either
\[
\models  p(\xx) \rightarrow \zeta(\yy)
\text{~~\ \ \ \ ~~~~or~~~\ \ \ \ ~~~}
\models p(\xx) \rightarrow \neg \zeta(\yy).
\]

By taking countable conjunctions, 
we see that every tuple in every
$L$-structure satisfies some complete
quantifier-free type (in the sense of Definition~\ref{newTypeDef}).
This justifies our use of the word \emph{type} for a single
formula.

We will often call complete quantifier-free types of $\Lwow(L)$ simply \emph{quantifier-free types}, and sometimes refer to them as \emph{quantifier-free $\Lwow(L)$-types}.  Although we have required that a quantifier-free type $p(\xx)$ have free variables precisely those in $\xx$, when there is little possibility of confusion we will sometimes omit the tuple $\xx$ and refer to the quantifier-free type as $p$.

We say that a quantifier-free type $p(\xx)$ is \defn{consistent with} a theory $T$
when $T \cup (\exists \xx) p(\xx)$ has a model.  A tuple $\aa$ in an
$L$-structure $\cM$, where $|\aa| = |\xx|$, is said to  \defn{realize} the
quantifier-free type $p(\xx)$ when $\cM\models p(\aa)$; in this case we say
that $p(\xx)$ is \emph{the} quantifier-free type of $\aa$ (as it is unique up to
equivalence).

Suppose that $p(\xx)$ and $q(\yy)$ are quantifier-free types, where $\yy$ is a tuple of variables containing those in $\xx$.  We say that $q$ \defn{extends} $p$, or that $p$ is the \defn{restriction} of $q$ to $\xx$, when $\models q(\yy)\to p(\xx)$.

\begin{definition}
A quantifier-free type $p(x_1, \ldots, x_n)$ is said to be \defn{non-redundant} when it implies the formula $\bigwedge_{1\le i< j\le n} (x_i \neq x_j)$.
\end{definition}

Note that every quantifier-free type is equivalent to the conjunction of a non-redundant
quantifier-free type and equalities of  variables, as follows.
Suppose $q$ is a quantifier-free type. Let $S$ be the set containing those formulas $(u=v)$,
for $u$ and $v$ free variables of $q$, such that $q$ implies $(u=v)$.
Then $q$ is equivalent to 
\[r \And \bigwedge_{\eta\in S} \eta\]
for some non-redundant quantifier-free type $r$.

The following notion will
be used in the 
stages of Construction~\ref{themainconstruction}
that we call ``refinement''.

\begin{definition}
\label{dupdef}
We say that a theory $T$ has \defn{duplication of quantifier-free types}
when, for every non-redundant quantifier-free type $p(x, \zz)$
consistent with $T$,
there
is a non-redundant quantifier-free type $q(x,y,\zz)$
consistent with $T$ such
that
\[
\models q(x,y,\zz) \rightarrow \bigl ( p(x,\zz) \And p(y, \zz)\bigr),
\]
where $p(y, \zz)$ denotes the quantifier-free type $p(x,\zz)$ with all instances of the
variable $x$ replaced by the variable $y$.
\end{definition}

Equivalently,
when $T$ does not have duplication of quantifier-free types, there is some non-redundant quantifier-free type $p(x,\zz)$ consistent with $T$ and some model $\cM$ of $T$ containing a tuple $(a,\bb)$ realizing $p$ such that the only way for a tuple $(a', \bb)$ in $\cM$ to also realize $p$ is for $a'$ to equal $a$.


\subsection{Construction of a Borel $L$-structure strongly witnessing a theory}\label{construction}
\ \\ \indent
Throughout \S\ref{construction}, let
$L$ be a countable \emph{relational} language
and
let $T$ be a countable pithy $\Pi_2$ theory of $\Lwow(L)$ that has
duplication of quantifier-free types.
We now construct a Borel $L$-structure $\PP$ that
strongly witnesses $T$.
This construction is inspired by
\cite[Theorem~5]{MR2724668}, in which Petrov and Vershik build an analogous
continuum-sized structure realizing the theory of the Henson graph
$\Henson$.
We begin with an informal description.

We construct $\PP$ by assigning, for every increasing tuple of reals,
the quantifier-free type that it realizes.  This will determine the quantifier-free type of every tuple of reals, and hence determine the structure $\PP$ on $\Reals$, as we now explain.

\begin{definition}
Given (strictly) increasing tuples of reals $\cc$ and $\dd$, we say that 
$\cc$ \defn{isolates} $\dd$
when every left-half-open interval whose endpoints are consecutive entries
of $\cc$ contains exactly one entry of $\dd$, i.e., for $0 \le j < \ell$, 
\[ d_j \in (c_j, c_{j+1}],\]
where $\cc = c_0\cdots c_{\ell}$ and $\dd = d_0\cdots d_{\ell-1}$.
\end{definition}
For example, the 
triple $(1,2,5)$ isolates the pair $(2,3)$.

Let $\cc$ be an arbitrary tuple of reals (not necessarily in increasing
order, and possibly with repetition). 
The quantifier-free type of 
$\cc$ 
is determined by the
quantifier-free type of 
the increasing tuple containing all reals of $\cc$
along with the relative ordering of the entries of
$\cc$.
For example, let $p(x,y,z)$ be the quantifier-free type of $(1,2,5)$. Then 
the quantifier-free type of $(2,2,1,5)$ is the unique (up to equivalence) formula
$q(y,w,x,z)$ implied by $p(x,y,z) \And (y=w)$.
Hence in order to determine the quantifier-free type of every tuple of
reals, and thereby define a structure
$\PP$ on $\Reals$, it suffices  to assign the
quantifier-free types of all increasing tuples.

In Construction~\ref{themainconstruction},
we will build our Borel $L$-structure $\PP$ inductively, making sure that
$\PP$
strongly witnesses $T$. At stage $i\ge 0$ of the
construction we will 
define
the following quantities:
\begin{itemize}
\item $\rr_i = (r^i_0, \ldots, r^i_{|\rr_i|-1})$, the increasing tuple of all rationals mentioned by the end of stage $i$;

\item $p_i$, the quantifier-free type of $\rr_i$; and

\item $\vv_i = (v^i_0, \ldots, v^i_{|\rr_i|})$, an increasing tuple of
irrationals 
that isolates
$\rr_i$.
\end{itemize}
We call the left-half-open intervals
\[
\bigl(-\infty, v^i_0\bigr],\, \bigl(v^i_0, v^i_1\bigr],\,
\ldots, \,\bigl(v^i_{|\rr_i|-1}, v^i_{|\rr_i|}\bigr], \,\bigl( v^i_{|\rr_i|},
\infty\bigr)
\]
the \emph{intervals
determined by $\vv_i$}.

We will 
define 
the $\vv_i$ so that 
they form
a nested sequence of tuples 
of irrationals such that
every (increasing) tuple 
that 
a given $\vv_i$ isolates,
including $\rr_i$, is assigned the same quantifier-free type $p_i$ at stage $i$.
The sequence of tuples $\{\vv_j\}_{j\in\omega}$ will be 
such that the set of reals $\bigcup_{j\in\omega} \vv_j$ is dense in
$\Reals$. Thus
for every tuple
of reals $\aa$, all of its entries 
occur
in some increasing tuple 
isolated by
$\vv_i$ for some $i$, and so its quantifier-free type will eventually be defined. 
This motivates the following definitions.

\begin{definition}
For each stage $i$, define $\B_i$ to be the set of tuples $\cc \in \Reals$ such that there is some increasing tuple $\dd \in \Reals$ that 
$\vv_i$ isolates and that
contains every entry of $\cc$. 
\end{definition}

Note that $\B_i \subseteq \B_{i'}$ for $i\le i'$, and that
$\bigcup_{j\in\omega} \B_j$ contains every tuple of reals.

By the end of stage $i$, we will have defined the quantifier-free type of every tuple that 
$\vv_{i}$
isolates, and hence by extension, of every tuple in $\B_i$.
For example, if $(1,2,5)\in\B_i$, then $(2,2,1,5)\in\B_i$ also,
and by the end of stage $i$ the
quantifier-free type of $(1,2,5)$ will be determined explicitly, and of $(2,2,1,5)$
implicitly, as described above.

Next we define an equivalence relation on $\B_i$, which we call
$i$-equivalence.
By the end of stage $i$, tuples in $\B_i$ that are $i$-equivalent
will have been assigned the same quantifier-free type.

\begin{definition}
Let $i\ge0$.  We say that two tuples $\cc, \dd \in\B_i$ of the same length $\ell$ are \defn{$i$-equivalent}, denoted $\cc \approx_i \dd$, if for all $j \le \ell$, the $j$th entry of $\cc$ and $j$th entry of $\dd$ both fall into the same interval determined by $\vv_i$.  Any two elements of $\B_i$ of different lengths are not $i$-equivalent.
\end{definition}

For example, each left-half-open interval $(v^i_j, v^i_{j+1}]$
determined by $\vv_i$, where
\linebreak $0 \le j < |\rr_i|$, is the $i$-equivalence class of any element
in the
interval.

Note that 
$i'$-equivalence 
refines $i$-equivalence 
for $i'> i$, in the sense that given two elements of $\B_i$ that are not $i$-equivalent, they are also not $i'$-equivalent.
Furthermore, our construction will be such that
for any two distinct $i$-equivalent tuples in $\B_i$, there is some
$i'> i$ for which
they are not $i'$-equivalent.

In the construction,
we will assign 
quantifier-free types in such a way that
$\PP$ strongly witnesses $T$. Specifically, for every $\aa\in
\B_i$ and
every pithy $\Pi_2$ sentence $(\forall
\xx)(\exists y)\psi(\xx,y) \in T$,
if there is no internal witness for $(\exists
y)\psi(\aa,y)$, then at some stage $i' > i$ we will build a left-half-open
interval $I$, disjoint from $\B_i$, consisting of $i'$-equivalent elements all of which are
external witnesses for
$(\exists y)\psi(\aa,y)$.
This will imply that for any 
$\cc \approx_{i'} \aa$, every $b\in I$ will
witness $(\exists y)\psi(\cc,y)$, since $\aa b \approx_{i'} \cc b$ and 
$i'$-equivalent tuples realize the same quantifier-free type.

We will build these external witnesses in the even-numbered
stages of the construction;
we call this process \emph{enlargement}, because we extend the portion of the real line to which we assign quantifier-free types.  In the odd-numbered stages, we perform \emph{refinement} of intervals, so that distinct $i$-equivalent tuples in $\PP$ are eventually not $i'$-equivalent for some $i'>i$; this ensures that each expression $(\exists y)\psi(\aa,y)$ will be witnessed with respect to all possible $\aa\in \PP$.
In fact, as we have noted earlier, by the end of stage $i$, we will have assigned the quantifier-free type of every tuple in $\B_i$ in such a way that $i$-equivalent tuples have the same quantifier-free type.

\begin{construction}
\label{themainconstruction}
Fix an enumeration  $\{\varphi_i\}_{i\in\w}$ of the sentences of
$T$ such that every sentence of $T$ appears infinitely often.
Because $T$ is a pithy $\Pi_2$ theory, for each $i$,  the sentence $\varphi_i$ 
is of the form
\[
(\forall \xx)(\exists y)\psi_{i}(\xx, y), 
\]
where $\psi_i$ is a quantifier-free formula whose free variables  are
precisely $\xx y$, all distinct, and where $\xx$ is possibly empty.  Consider the induced enumeration  $\{\psi_i\}_{i\in\w}$, and for each $i$, let $k_i$ be one less than the number of free variables of $\psi_i$.
Also fix an enumeration $\{q_i\}_{i \in\w}$ of the rationals.
\end{construction}
\vspace{-3pt}
We now give the inductive construction.
For a diagram, see Figure~\ref{illustration}.
The key inductive property is that at the end of each stage $i$, the
quantifier-free type
$p_i$ is consistent with $T$, extends $p_{i-1}$ (for $i\ge 1$), and  
is the (non-redundant) quantifier-free type of every tuple that 
$\vv_i$ isolates,
including $\rr_i$.

\vspace{5pt}
\noindent{\textbf{Stage $\mathbf{0}$:}}
Set $\rr_{0}$ to be the tuple $(0)$, let $p_{0}$ be any quantifier-free unary
type consistent with $T$, and set $\vv_{0}$ to be the pair
$\bigl (-\sqrt{2}, \sqrt{2}\bigr)$.

\vspace{5pt}
\noindent {\textbf{Stage $\mathbf{2i+1}$ (Refinement):}}
In 
stage $2i+1$,
we will construct a tuple $\rr_{2i+1}$ of rationals, a tuple
$\vv_{2i+1}$ of irrationals, and a non-redundant quantifier-free type $\pt$
consistent with $T$ in such a way that these extend $\rr_{2i}$, $\vv_{2i}$, and $p_{2i}$, respectively.
In doing so, we will refine the intervals determined by $\vv_{2i}$, and assign the
quantifier-free type of every increasing tuple that 
$\vv_{2i+1}$ isolates. 
By extension, this will determine the quantifier-free type of every
tuple in $\B_{2i+1}$, i.e.,
of
every tuple $\cc$ 
all of whose entries are contained in some tuple that 
$\vv_{2i+1}$ isolates,
in such a way that $(2i+1)$-equivalent tuples are assigned the same quantifier-free
type.

Define $\rr_{2i+1}$  to be the increasing tuple consisting of the
entries of $\rr_{2i}$ along with $q_i$.
We need to
define the quantifier-free type $p_{2i+1}$ of  $\rr_{2i+1}$  so that it extends $p_{2i}$, the quantifier-free type of $\rr_{2i}$.
There are three cases, depending on the value of $q_i$.

\textbf{Case 1:} The ``new'' rational $q_i$ is already an entry of $\rr_{2i}$. In this case,
there is nothing to be done, as 
$\rr_{2i+1} = \rr_{2i}$, and so we set
$p_{2i+1} \defas p_{2i}$.

\textbf{Case 2:}  
We have $q_i \in (-\infty, v^{2i}_0]\, \cup\, (v^{2i}_{|\rr_{2i}|},
\infty)$,
i.e.,
the singleton tuple $(q_i)$ is not in $\B_{2i}$.
If $q_i < v^{2i}_0$, 
let $p_{2i+1}(x_0, \dots, x_{|\rr_{2i}|})$ be any 
non-redundant
quantifier-free type consistent with $T$ that implies $p_{2i}(x_1, \dots,
x_{|\rr_{2i}|})$. Similarly,
if $q_i > v^{2i}_{|\rr_{2i}|}$,
let
$p_{2i+1}(x_0, \dots, x_{|\rr_{2i}|})$ 
be such that it implies $p_{2i}(x_0, \dots, x_{|\rr_{2i}|-1})$.
Such quantifier-free types $p_{2i+1}$ must exist, because 
$p_{2i}$ is consistent with $T$ and we have not yet
determined the set of relations that hold of any
tuple that has an entry lying outside the interval
$(v^{2i}_0,  v^{2i}_{|\rr_{2i}|}]$.

\textbf{Case 3:} Otherwise.
Namely, $q_i \in (v^{2i}_j , v^{2i}_{j+1}]$
for some $j$ such that $0\le j < |\rr_{2i}|$, and $q_i \neq
r^{2i}_j$. 
Note that 
$\vv_{2i}$ isolates each of
the tuples 
\[
r^{2i}_0 \cdots r^{2i}_{|\rr_{2i}|-1} \qquad \text{and} \qquad
r^{2i}_0\cdots r^{2i}_{j-1} \,q_i\, r^{2i}_{j+1}\cdots r^{2i}_{|\rr_{2i}|-1}
\]
and
hence 
by our construction,
the tuples both
satisfy the same quantifier-free type $p_{2i}$. 
By our assumption that $T$ has duplication of quantifier-free types, and because $p_{2i}$
is non-redundant, there must be a
non-redundant quantifier-free type
$p_{2i+1}(x_0, \dots, x_{|\rr_{2i}|})$ 
consistent with $T$ that implies 
\[
p_{2i}(x_0, \ldots, x_j, x_{j+2}, \ldots, x_{|\rr_{2i}|})
\And p_{2i}(x_0, \ldots, x_{j-1}, x_{j+1}, \ldots, x_{|\rr_{2i}|}).
\]

Whichever case holds,
now let $\vv_{2i+1}$ be any increasing tuple of irrationals that contains
every entry of $\vv_{2i}$ and 
isolates
$\rr_{2i+1}$.  Consider the
subtuple of variables $\zz \subseteq x_0 \cdots x_{|\rr_{2i}|}$ that
corresponds to the positions of the entries of $\rr_{2i}$ within
$\rr_{2i+1}.$ In each case above, $p_{2i+1}(x_0, \ldots, x_{|\rr_{2i}|})$ is a non-redundant quantifier-free type consistent with $T$ whose restriction to $\zz$ is $p_{2i}(\zz)$.  Hence we may assign the quantifier-free type of every increasing tuple that 
$\vv_{2i+1}$ isolates
(including $\rr_{2i+1}$) to be $p_{2i+1}$.  By extension, this determines the quantifier-free type of every tuple in $\B_{2i+1}$.

\vspace{5pt}
\noindent {\bfseries Stage $2i+2$ (Enlargement)}:
In stage $2i+2$, we will construct a 
tuple $\rr_{2i+2}$ of rationals, a tuple
${\vv}_{2i+2}$ of irrationals, and a non-redundant quantifier-free type $p_{2i+2}$
consistent with $T$ 
in such a way that these
extend $\rr_{2i+1}$,
$\vv_{2i+1}$, and $\pt$, respectively.
As we do so, we will enlarge the portion of the real line 
to which we assign quantifier-free types. At the end of the stage we will have
determined the quantifier-free type of every tuple in $\B_{2i+2}$, in such
a way that $(2i+2)$-equivalent tuples are assigned the same quantifier-free
type.

Our goal is to provide witnesses for $(\exists y) \psi_i(\aa, y)$,
where $\psi_i(\xx,y)$ is from our enumeration above,
for each $k_i$-tuple of reals
$\aa\in\B_{2i+1}$, i.e., each $k_i$-tuple whose quantifier-free type is
determined by $\pt$.
We extend the non-redundant quantifier-free type $\pt(x_0, \ldots, x_{|\rr_{2i+1}|-1})$ to a
non-redundant quantifier-free type $p_{2i+2}(x_0, \ldots,$  $x_{|\rr_{2i+1}|-1}, \ww)$ 
so that 
for every 
tuple $\aa$ such that 
$(\exists y) \psi_i(\aa, y)$ has no
internal witness, 
there is an entry of $\ww$
whose realizations provide external witnesses.

Let $\{\zz_\ell\}_{1\le \ell \le N_i}$ be an enumeration of those tuples of variables 
(possibly with repetition)
of length $k_i$
all of whose entries are among $x_0, \ldots, x_{|\rr_{2i+1}|-1}$.
We now define, by induction on $\ell$, intermediate non-redundant
quantifier-free types 
\[
s_\ell(x_0, \ldots, x_{|\rr_{2i+1}|-1}, \uu_\ell),
\] 
for $0 \le \ell < N_i$, that are
consistent with $T$ and
such that each $s_{\ell+1}$ implies $s_\ell$.
As we step through the tuples of variables of length $k_i$, if we have
already provided a ``witness'' for
$(\exists y) \psi_i(\zz_{\ell+1}, y)$  then we will do nothing; otherwise, we will extend
our quantifier-free type to provide one, as we now describe.

Let $s_0 \defas p_{2i+1}$, and let $\uu_0$ be the empty tuple of
variables. Now consider step \linebreak $\ell < N_i$ of the induction, so
that $s_0, \ldots, s_\ell$ have
been defined. If there is a variable $t$ among $x_0 \cdots
x_{|\rr_{2i+1}|-1}$ or among $\uu_\ell$
such that 
$s_\ell(x_0, \ldots, x_{|\rr_{2i+1}|-1}, \uu_\ell)$
implies $\psi_i(\zz_{\ell+1}, t)$, then let $s_{\ell+1} \defas s_\ell$ and
$\uu_{\ell+1} \defas \uu_\ell$.
If not, then because $s_\ell$ is consistent with $T$ and 
$(\forall \xx)(\exists y)\psi_i(\xx, y) \in T$,
there must be some non-redundant quantifier-free type $s_{\ell+1}$ consistent with $T$  
that has one more variable, $w_{\ell+1}$, than $s_\ell$, 
such that $s_{\ell+1}$ implies both $s_\ell$ and $\psi_i(\zz_{\ell+1},
w_{\ell+1})$; in this
case, let $\uu_{\ell+1} \defas \uu_\ell w_{\ell+1}$.
Let 
\[
p_{2i+2}(x_0, \ldots, x_{|\rr_{2i+1}|-1}, \ww) \defas s_{N_i}(x_0, \ldots,
x_{|\rr_{2i+1}|-1}, \ww),
\]
where 
$\ww \defas \uu_{N_i}$.
Note that $p_{2i+2}(x_0, \ldots, x_{|\rr_{2i+1}|-1}, \ww)$ 
is a non-redundant quantifier-free type that is consistent with $T$
and
extends $p_{2i+1}(x_0, \ldots, x_{|\rr_{2i+1}|-1})$.

Next, choose $|\ww|$-many rationals greater than all entries of $\rr_{2i+1}$, and
define $\rr_{2i+2}$ to be the increasing tuple consisting of
$\rr_{2i+1}$ and these new rationals.
Let $\vv_{2i+2}$ be an arbitrary increasing tuple of irrationals
that contains every entry of
$\vv_{2i+1}$ and 
isolates
$\rr_{2i+2}$.
Finally,  
for every increasing tuple that 
$\vv_{2i+2}$ isolates
(including
$\rr_{2i+2}$), declare its quantifier-free type to be $p_{2i+2}$.
As with the refinement stages, this determines by extension the quantifier-free type of every tuple in $\B_{2i+2}$,
i.e., of every tuple $\cc$ 
all of whose entries are contained in some tuple that 
$\vv_{2i+2}$ isolates.
In particular, for any tuple 
$\aa \in \B_{2i+1}$ 
of length $k_i$ such that $(\exists y)\psi_i(\aa,y)$ does not have internal witnesses,
we have constructed a left-half-open interval $(v^{2i+2}_j,
v^{2i+2}_{j+1}]$, for some $j$ such that $0 \le j < |\rr_{2i+2}|$,
consisting of external witnesses for
$(\exists y)\psi_i(\aa,y)$.
This ends the stage, and the construction.
\hfill\qed

\begin{figure}[h]
{
\hspace{-20pt}\!\!
\includegraphics[width=1.05\linewidth]{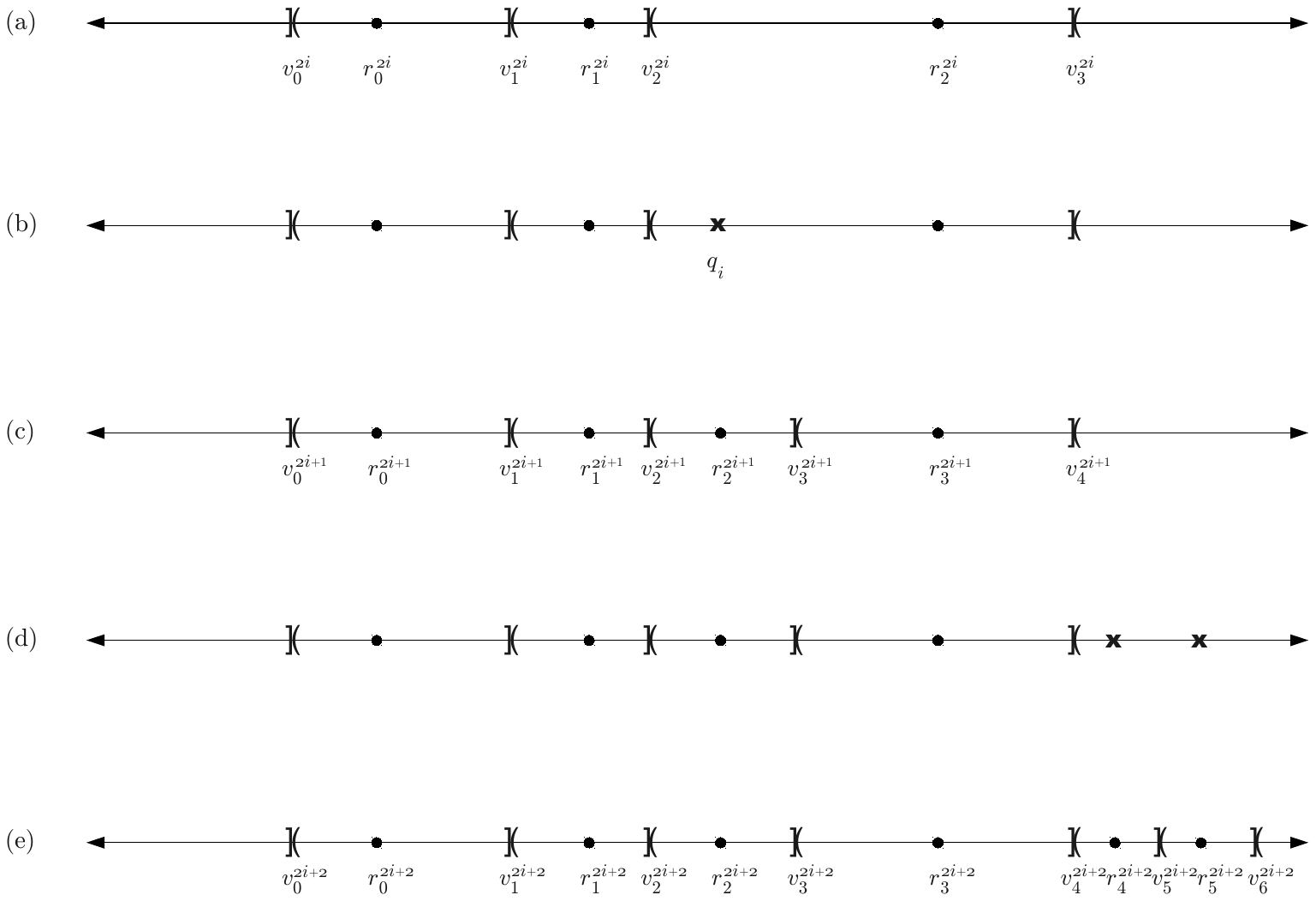}}
\caption{An illustration of Construction~\ref{themainconstruction}.\newline
(a) Suppose that we start stage $2i+1$ with the tuple
$\rr_{2i} = \bigl(r^{2i}_0,
r^{2i}_1, r^{2i}_2\bigr)$ of rationals and the tuple $\vv_{2i} = \bigl(v^{2i}_0, v^{2i}_1,
v^{2i}_2, v^{2i}_3\bigr)$ of
irrationals.\newline
(b) Suppose that the rational $q_i$ falls between $v^{2i}_2$ and $r^{2i}_2$.\newline
(c) By the end of stage $2i+1$, the rational $q_i$ has become $r^{2i+1}_2$, and the
rationals and irrationals to its right are reindexed.\newline
\indent (d) Suppose that in stage $2i+2$ we need two intervals of external witnesses 
for $(\exists y)\psi_i(\aa, y)$ 
as $\aa$ ranges among $k_i$-tuples all of whose entries are
entries of
$\rr_{2i+1}$. 
Then we select two new rational witnesses
$r^{2i+2}_4$ and $r^{2i+2}_5$  to the right of $\vv_{2i+1}$.
\newline
\indent (e) 
At the end of stage $2i+2$, we choose 
irrational interval boundaries $v^{2i+2}_5$ (between $r^{2i+2}_4$ and~$r^{2i+2}_5$)
and $v^{2i+2}_6$ (to the right of $r^{2i+2}_5$).
}
\label{illustration}
\end{figure}

We now verify that this construction produces a structure with the desired
properties.

\begin{theorem}
\label{ModelsWitnessingT}
Let $L$ be a countable relational language, 
let $T$ be a countable pithy $\Pi_2$ theory of $\Lwow(L)$ that has
duplication of quantifier-free types, and let $\PP$ be the 
$L$-structure obtained via Construction~\ref{themainconstruction}.
Then $\PP$ is 
a samplable Borel $L$-structure 
that strongly witnesses $T$.
\end{theorem}
\begin{proof}


We first show that  $\PP$ is a Borel $L$-structure.
Fix a relation symbol $R \in L$, and let $k$ be the arity of
$R$.
We must show that $\{\aa\in\PP^k \st \PP \models R(\aa)\}$
is Borel.
For each 
$i \ge 0$ define
\[
\X_i \defas \bigl\{\aa\in\PP^k \st \PP\models R(\aa) \text{~and~} \aa \in \B_i
\bigr\}.
\]
Recall that by our construction, the set of reals $\bigcup_{j\in\omega}\vv_j$ is
dense in  
$\Reals$, and so every tuple of reals is in $\bigcup_{j\in\omega} \B_j$.
Therefore 
\[
\textstyle
\{\aa\in \PP^k \st \PP \models R(\aa)\} = \bigcup_{j \in \omega} \X_j.
\]
 In particular, it suffices to show that $\X_j$ is Borel for each $j$.  

Fix some $i\ge 0$, and note that 
for every $\aa,\aa' \in\B_i$ such that  $\aa \approx_i \aa'$,  we have 
\[
\PP \models R(\aa) 
\qquad \text{if and only if} \qquad
\PP \models R(\aa'),
\]
because $i$-equivalent tuples are assigned the same quantifier-free
type.
Furthermore,  for every $\aa\in \B_i$, the set
\[
\{ \cc \in  \B_i \st \cc \approx_i \aa \}
\]
is a
$k$-fold product of left-half-open intervals. 
Hence $\X_i$ is Borel,
and so $\PP$ is a Borel $L$-structure.
Moreover, $\PP$ is samplable because $L$ is relational.


We now show that $\PP$ strongly witnesses $T$.
Let $m$ be an arbitrary nondegenerate
probability measure on $\Reals$.
Consider a pithy $\Pi_2$ sentence
\[
(\forall \xx)(\exists y)\xi(\xx,y) \in T,
\]
and let $\aa$ be a tuple of reals such that $|\aa| = |\xx|$, where
$|\xx|$ could possibly be zero.

Suppose $(\exists y)\xi(\aa, y)$ does not have an internal witness. 
Let $\ell^*$ be the least stage such that $\aa \in \B_{2 \ell^* + 1}$.
Because each sentence of $T$ appears infinitely often in the enumeration
$\{\varphi_j\}_{j\in\omega}$, 
there is some $\ell\ge \ell^*$ such that 
$\varphi_\ell = (\forall \xx)(\exists y)\xi(\xx,y)$, and hence such
that
$\xi = \psi_\ell$.

Since $\aa\in\B_{2\ell+1}$, at stage $2\ell+2$ there is some real $b$ such that
$\PP \models \psi_\ell(\aa,b)$. 
Furthermore, we have ensured that there is a left-half-open interval
of reals
$b'$ such that 
$b' \approx_{2\ell+2} b$
and hence 
such that $\PP \models \psi_\ell(\aa,b').$ 
Because $m$ is nondegenerate, this collection of
external witnesses for $(\exists y) \xi(\aa,y)$ has
positive $m$-measure.
Hence $(\PP,m)$ witnesses $T$. As $m$ was an arbitrary nondegenerate
probability measure on $\Reals$, the Borel $L$-structure $\PP$ strongly witnesses $T$, as desired.
\end{proof}

\subsection{Invariant measures from trivial definable closure}\label{invariant}\ \\
\indent 
We are now ready to prove the positive direction of our main theorem, Theorem~\ref{maintheorem}.
We have seen, in \S\ref{construction},
that if a countable pithy $\Pi_2$ theory
$T$ in a countable relational language $L$ has duplication of quantifier-free types, then
there exists a samplable Borel $L$-structure strongly witnessing $T$. 
	We show below that when
	$T$ has a unique countable model $\cM$ (up to isomorphism),
	duplication of quantifier-free types is moreover implied by $\cM$ having trivial definable closure;
	we prove the converse for relational languages in Corollary~\ref{partialconverse}.
	In fact, we have the following stronger result.

\begin{lemma}\label{alephnought}
Let
$T$ be a countable theory of $\Lwow(L)$ such that every countable model of $T$ has trivial definable closure. Then $T$ has duplication of quantifier-free types.
\end{lemma}
\begin{proof}
Suppose $p(x,\zz)$ is a non-redundant quantifier-free $\Lwow(L)$-type
consistent with $T$.  
Because every model of $T$ has trivial definable closure, 
\[
T \models p(x,\zz) \to (\exists y) \bigl( p(y,\zz) \wedge (y \neq x) \bigr).
\]
Hence there is some non-redundant quantifier-free $\Lwow(L)$-type
$q(x,y,\zz)$ such that
\[
T \models q(x,y,\zz) \to \bigl( p(x,\zz)  \wedge p(y,\zz) \bigr),
\]
and so $T$ has duplication of quantifier-free types.
\end{proof}

We now use 
Theorem~\ref{ModelsWitnessingT} and
Lemma~\ref{alephnought} 
to 
prove
the positive direction of
Theorem~\ref{maintheorem}.
\begin{theorem}
\label{InvariantMeasuresConcentrated}
Let $L$ be a countable 
language
and
let
$\cM$ be a countable infinite $L$-structure.
If $\cM$ has trivial definable closure, then there is an invariant probability measure 
on $\Model$ that is concentrated on the isomorphism class of $\cM$ in $\Model$.
\end{theorem}
\begin{proof}
Without loss of generality, we may assume that $\cM\in\Model$.
Let $\Mbar$ be the canonical structure of $\cM$ and $L_\Mbar$ its canonical
language.
By Proposition~\ref{pithypitwotheory-canonical},
there is a pithy $\Pi_2$ $\Lwow(L_\Mbar)$-theory  $T_\Mbar$ 
all of whose countable models are isomorphic
to $\Mbar$. 
By Lemmas~\ref{canon-is-interdefinable} and 
~\ref{trivial-dcl-interdefinability} and the fact that
$\cM$  has trivial definable closure, 
the unique 
(up to isomorphism)
countable model 
$\Mbar$ of $T_\Mbar$ has trivial definable closure.
Hence by Lemma~\ref{alephnought}, the theory $T_\Mbar$ has duplication of quantifier-free types.

Since $L_\Mbar$ is relational, by
Theorem~\ref{ModelsWitnessingT}
there is 
a samplable Borel
$L_{\Mbar}$-structure
$\QQ$ strongly witnessing $T_\Mbar$.
Therefore, by Corollary~\ref{stronglywitnessingcorollary}
there is an invariant probability
measure 
on $\Str_{L_\Mbar}$ that is concentrated 
on the set of countable $L_\Mbar$-structures that are isomorphic to $\Mbar$, i.e., those that are
models of $T_\Mbar$.
Finally, by Lemmas~\ref{canon-is-interdefinable} and \ref{invmeasinter},
there is an invariant probability measure 
on $\Model$ that is
concentrated on $\cM$.
\end{proof}

Although the proof of Theorem~\ref{InvariantMeasuresConcentrated} produces
a samplable Borel $L_\Mbar$-structure, where $L_\Mbar$ 
may be quite different from $L$ (in particular, $L_\Mbar$ is always infinite and relational),
we can obtain 
essentially the same invariant measure via a samplable Borel $L$-structure.
We will use this fact in \S\ref{graphlimits}.

\begin{corollary}
\label{reduct}
Let $L$ be a countable language 
and let 
$\cM$ be a countable infinite $L$-structure.
If $\cM$ has
trivial definable closure, then there is a samplable Borel $L$-structure $\PP$ such that
for any continuous nondegenerate probability measure $m$ on $\Reals$, the invariant measure 
$\mu_{(\PP,m)}$ is concentrated on the isomorphism class of $\cM$ in $\Model$.
\end{corollary}
\begin{proof}
Let $m$ be an arbitrary continuous nondegenerate probability measure on
$\Reals$, and
let $\Mbar$, $L_\Mbar$, $T_\Mbar$,
and $\QQ$ be as in the proof of
Theorem~\ref{InvariantMeasuresConcentrated}.
In particular, $\QQ$ strongly witnesses $T_\Mbar$, and so by 
Theorem~\ref{BorelLStructuresWitnessingTLeadToInvariantMeasures},
the invariant measure
$\mu_{(\QQ,m)}$ is concentrated on $\Mbar$.

Note that,
by Lemma~\ref{SubmodelsSatisfyingT},
$\QQ$ and $\Mbar$ have the same $\Lwow(L_\Mbar)$-theory.
By 
Lemma \ref{canon-is-interdefinable}, $\cM$ and $\Mbar$ are interdefinable;
let $(\Psi_0, \Psi_1)$ be an interdefinition
between $\cM$ and $\Mbar$.
As in Lemma~\ref{interdefinition-lemma},
let $\PP$ be the
$L$-structure that is
interdefinable with $\QQ$ via
the interdefinition $(\Psi_0, \Psi_1)$,
so that $\PPbar = \QQ$, using the notation described after Lemma~\ref{canon-is-interdefinable}.
In particular, $\PP$ and $\cM$
have the same $\Lwow(L)$-theory.

Let $\varphi(\xx)$ be an arbitrary $\Lwow(L_\Mbar)$-formula.
By Lemma~\ref{qfdef-canonical}, 
there is some quantifier-free $\Lwow(L_\Mbar)$-formula $\psi_\varphi(\xx)$ such that 
\[
\Mbar \models \ \varphi(\xx) \leftrightarrow \psi_\varphi(\xx).
\]
Because $\PPbar$ and $\Mbar$ have the same $\Lwow(L_\Mbar)$-theory,
we also have
\[
\PPbar \models \ \varphi(\xx) \leftrightarrow \psi_\varphi(\xx).
\]
Hence any countable substructure of $\PPbar$ isomorphic to $\Mbar$ must in fact be an $\Lwow(L_\Mbar)$-elementary substructure of $\PPbar$.

As every $\Lwow(L_\Mbar)$-definable set in $\PPbar$ is equivalent to a quantifier-free definable set, every definable set in $\PPbar$ is Borel. However every $\Lwow(L)$-definable set in $\PP$ is an $\Lwow(L_\Mbar)$-definable set in $\PPbar$ and so every $\Lwow(L)$-definable set in $\PP$ is Borel. Hence
for any subbasic formula $\zeta$ of $\Lwow(L)$, the set in $\PP$ defined by $\zeta$ is Borel.
Thus $\PP$ is a Borel $L$-structure.

Because $\cM$ has trivial definable closure, $L$ has no constant symbols and every function of $\cM$ is a choice function. Since $\PP$ satisfies the same $\Lwow(L)$-theory as $\cM$, every function of $\PP$ is also a choice function. 
Hence $\PP$ is samplable.

By Lemmas~\ref{mu-is-invariant} and \ref{continuouslemma}, we have that
$\mu_{(\PP, m)}$ is an invariant probability measure on $\Model$ that is concentrated on the union of isomorphism classes of countable infinite substructures of $\PP$.
It remains to show that $\mu_{(\PP, m)}$ is concentrated on the isomorphism class of $\cM$, using the fact that $\mu_{(\PPbar, m)}$ is concentrated on the isomorphism class of $\Mbar$.

Suppose a countable infinite set $N\subseteq \Reals$ is such that the substructure 
$\cN^*$ 
of $\PPbar$ having
underlying set $N$ is isomorphic to $\Mbar$.
We will show that the substructure $\cN^\times$ of $\PP$ having underlying set $N$ is isomorphic to $\cM$.
Again as in Lemma~\ref{interdefinition-lemma},
let $\cN$ be the
$L$-structure that is
interdefinable with $\cN^*$ via
the interdefinition $(\Psi_0, \Psi_1)$,
so that $\Nbar  = \cN^*$. As noted above, $\Nbar$ is an $\Lwow(L_\Mbar)$-elementary substructure of $\PPbar$.
Therefore for any $\Lwow(L_\Mbar)$-formula $\varphi(\xx)$, 
\[
\label{phieqn}
\bigl\{ \aa \in\Nbar \st \PPbar \models \varphi(\aa) \bigr\}
=
\bigl\{ \aa \in\Nbar \st \Nbar \models \varphi(\aa) \bigr\}.
\tag{$\ddag$}
\] 
By \eqref{phieqn} and the fact that $(\Psi_0, \Psi_1)$ is an interdefinition between $\PP$ and $\PPbar$, and between $\cN$ and $\Nbar$, we have
\begin{eqnarray*}
\bigl\{ \aa \in \cN^\times \st \PP \models \psi(\aa)\bigr\}
&=&
\bigl\{ \aa \in \Nbar \st \PPbar \models \Psi_0(\psi)(\aa)\bigr\} \\
&=&
\bigl\{ \aa \in \Nbar \st \Nbar \models \Psi_0(\psi)(\aa)\bigr\} \\
&=&
\bigl\{ \aa \in \cN \st \cN \models \Psi_1(\Psi_0(\psi))(\aa)\bigr\} \\
&=&
\bigl\{ \aa \in \cN \st \cN \models \psi(\aa)\bigr\}
\end{eqnarray*}
for every $\Lwow(L)$-formula $\psi(\xx)$. 
Further, as $\cN^\times$ is a substructure of $\PP$, 
for every \emph{quantifier-free} $\Lwow(L)$-formula $\psi(\xx)$,
we have
\begin{eqnarray*}
\bigl\{ \aa \in \cN^\times \st \cN^\times \models \psi(\aa)\bigr\}
&=&
\bigl\{ \aa \in \cN^\times \st \PP \models \psi(\aa)\bigr\},
\end{eqnarray*}
and so
\begin{eqnarray*}
\bigl\{ \aa \in \cN^\times \st \cN^\times \models \psi(\aa)\bigr\}
&=&
\bigl\{ \aa \in \cN \st \cN \models \psi(\aa)\bigr\}.
\end{eqnarray*}
Hence $\cN = \cN^\times$, and so $\cN^\times$ is isomorphic to $\cM$, as $\Nbar$ is isomorphic to $\Mbar$.

By the fact that $\mu_{(\PPbar,m)}$ is concentrated on the isomorphism class of $\Mbar$ in $\Str_{L_\Mbar}$, we have 
		\[
			m^\infty \Bigl ( \bigl\{ A \in \Reals^\omega \st \F_\PPbar(A) \cong \Mbar \bigr\}\Bigr ) = 1,
		\]
		as 
$\mu_{(\PPbar, m)}  = m^{\infty}\circ \F_{\PPbar}^{-1}$.
Now suppose $A\in \Reals^\omega$ is such that $\F_\PPbar(A) \cong \Mbar$, and consider the the set $N\subseteq\Reals$ of entries of $A$.  By the above, the substructure $\cN^\times$ of $\PP$ having underlying set $N$ is isomorphic to $\cM$,
	and so $\F_\PP(A) \cong \cM$.
Hence 
\[
		 \bigl\{ A \in \Reals^\omega \st \F_\PPbar(A) \cong \Mbar \bigr\} \subseteq
	 \bigl\{ A \in \Reals^\omega \st \F_\PP(A) \cong \cM \bigr\} ,
	 \] 
	 and so
		\[
			m^\infty \Bigl ( \bigl\{ A \in \Reals^\omega \st \F_\PP(A) \cong \cM \bigr\}\Bigr ) = 1.
		\]
Therefore $\mu_{(\PP, m)}$ is concentrated on
the isomorphism class of $\cM$ in $\Model$.
\end{proof}


\section{Non-existence of invariant measures}\label{nonexistence}
In this section we complete the proofs of Theorems~\ref{maintheorem} and \ref{maintheoremthree}.
We begin by considering the converse of
Theorem~\ref{InvariantMeasuresConcentrated},
namely, that for any countable language $L$, a countable infinite
$L$-structure 
having nontrivial definable closure cannot admit an invariant measure.

Suppose a countable $L$-structure $\cM$ admits an invariant measure.
If there exists an
element $b\in\dcl_{\cM}(\emptyset)$, then 
for every $n\in\Nats$ 
the measure assigns the same positive probability to the event that $n$ 
satisfies the quantifier-free type 
of $b$,
which is not possible. More generally, 
$\cM$ having non-trivial definable closure leads to a contradiction, as we
show below;
a special case of this has been observed in \cite[(4.29)]{MR1066691}.

In fact, an even more general result holds.
Upon taking the special case of $G = \sym$, Theorem~\ref{KMConditions}
below completes the proof of our main result
by establishing that property (1) implies property (2) in
Theorem~\ref{maintheorem}.
Indeed, initially we proved  only this special case.
However, Alexander Kechris and
Andrew Marks noticed that with minor modifications to our original proof,
the stronger result Theorem~\ref{KMConditions} holds.
We have included this with their permission.


\begin{theorem}
\label{KMConditions}
Let $L$ be a countable language, and let $\cM$ be a countable infinite $L$-structure.  
Suppose that $G \le \sym$ is the automorphism group of a structure in
$\Model$ that has trivial definable closure.
If there is a $G$-invariant 
probability measure $\mu$ on $\Model$ that is concentrated on the isomorphism class of $\cM$, then $\cM$ must have trivial definable closure.
\end{theorem}
\begin{proof}
Without loss of generality, we may assume that $\cM\in\Model$.
Suppose, for a contradiction,
that there is a tuple $\aa\in\cM$
and an
element $b\in\cM$ such that $b \in \dcl_{\cM}(\aa)-\aa$.
By considering the Scott sentence of the structure obtained by expanding $\cM$ by constants for the tuple $\aa$
(see, e.g., \cite[Theorem~3.3.5]{MR2062240}),
we can find a formula $p(\xx, y)\in \Lwow(L)$ such that
$\cM\models p(\aa, b)$ and whenever $\cM \models p(\cc, d)$ there is an
automorphism of $\cM$ taking $\aa b$ to $\cc d$ pointwise.

In particular,
\[\cM \models (\exists \xx y) p(\xx, y).\]
Since the measure $\mu$ is concentrated on $\cM$, hence on $L$-structures that satisfy 
$(\exists \xx y )p(\xx, y)$,
we have 
\[\mu\bigl(\llrr{(\exists \xx y) p(\xx, y)}\bigr) = 1.\]
By the countable additivity of $\mu$,
there is some tuple
$\mm n \in \Nats$ such that
$\mu\bigl(\llrr{p(\mm, n)}\bigr)>0$.
Fix such an $\mm n$ and let 
$\alpha\defas \mu\bigl(\llrr{p(\mm, n)}\bigr)$.

Note that if 
$\cM \models p(\aa, c)$
for some  $c\in\Nats$, then there is an automorphism of $\cM$ that fixes
$\aa$ pointwise and sends $b$ to $c$;
hence, as $b\in\dcl(\aa)$, such an element $c$ is equal to $b$.
Therefore
we have
\[
\cM \models (\forall \xx y_1 y_2) \bigl( 
(
p(\xx, y_1) \wedge p(\xx, y_2)
)
 \to  (y_1 = y_2)
\bigr).
\]
Hence
for $j, k \in\Nats$,
\[
\llrr{p(\mm, j)} \cap \llrr{p(\mm, k)} = \emptyset
\]
whenever $j \neq k$, and so
\[
\mu\bigl(\llrr{(\exists y) p(\mm, y)} \bigr)
=
\sum_{j\in\Nats}
\mu\bigl(\llrr{p(\mm, j)}\bigr)
.
\]

Now let 
\[X \defas \bigl\{ j\in\Nats \st  (\exists g \in G) \, \bigl ((g(\mm) = \mm) \wedge
(g(n) = j) \bigr)
\bigr\}.\]
By the $G$-invariance of $\mu$,
for all $j \in X$ we have
\[
\mu\bigl( \llrr{p(\mm, j)} \bigr)= \alpha.
\]
However, because $G$ is the automorphism group of a
structure having trivial definable closure, $X$ is infinite.
But 
\[1 \ge \mu\bigl (\llrr{(\exists y) p(\mm, y)}\bigr)
\ge 
\sum_{j\in X}
\mu\bigl(\llrr{p(\mm, j)}\bigr)
=  
\sum_{j\in X} \alpha
,
\]
which is a contradiction as $\alpha > 0$ and $X$ is infinite.
\end{proof}

Consider the properties (1), (2), and (3) of Theorems~\ref{maintheorem} and \ref{maintheoremthree}.
Theorem~\ref{KMConditions} establishes that (3) implies (2), and 
Theorem~\ref{InvariantMeasuresConcentrated} shows that (2) implies (1).
Finally, (1) trivially implies (3), by considering 
the unique countable infinite structure in the empty language, which has
automorphism group $\sym$. Hence the proof of Theorem~\ref{maintheoremthree} is complete.

Kechris and Marks also observed that the equivalence of properties (2)
and (3) from Theorem~\ref{maintheoremthree} can be established more easily
than the full theorem, in particular while
entirely avoiding the machinery of Section~\ref{existence}.
We again include the argument with their permission.

\begin{corollary}[Kechris--Marks]
\label{KMCor}
Let $L$ be a countable language, and let $\cM$ be a countable infinite
$L$-structure.
The following are equivalent:
\begin{itemize}
\item[(2)]
The structure $\cM$ has trivial
group-theoretic definable closure, i.e.,
for every finite tuple $\aa\in\cM$, we have
$\dcl_{\cM}(\aa)= \aa$, where
$\dcl_{\cM}(\aa)$ is the collection of elements $b\in\cM$ that are fixed by
all automorphisms of $\cM$ fixing $\aa$ pointwise.
\item[(3)] 
There is some $\cN\in\Model$ that has trivial group-theoretic definable closure and is such that
there is an
$\Aut(\cN)$-invariant
probability measure on
$\Model$
concentrated on the 
set of elements of $\Model$ that are isomorphic to
$\cM$.
\end{itemize}
\end{corollary}
\begin{proof}
We already have that (3) implies (2) by Theorem~\ref{KMConditions}. But that (2)
implies (3) follows by taking $G = \Aut(\cM)$ and $\mu$ to be a Dirac delta
measure on the structure $\cM$ itself (where we again take $\cM$ to be in $\Model$ without loss of generality).
\end{proof}

As an immediate corollary of Theorem~\ref{KMConditions}, we see that any countable infinite structure 
that admits an
invariant measure cannot have constants, and all of its functions must be choice functions.
This observation has been used in \cite{AckermanSemigroups} to classify those 
commutative $n$-semigroups 
as well as those
ultrahomogeneous $n$-semigroups 
that admit an invariant measure.
\begin{corollary}
\label{selector}
Let $L$ be a countable language, and let $\cM$ be a countable infinite $L$-structure.  Suppose that either $L$ has constant symbols or that there is a function symbol $f\in L$ and tuple $\aa\in\cM$ for which $f^\cM(\aa)\not\in\aa$.  Then there is no invariant probability measure 
on $\Model$ 
that is
concentrated on the isomorphism class of $\cM$.
\end{corollary}

Note that this implies that 
Corollary~\ref{maincorollary}, which characterizes
those \Fr\ limits 
in relational languages
that admit invariant measures,
does not extend to structures with constants
or functions.  This is demonstrated, e.g., by Hall's countable universal locally
finite group, whose age has the strong amalgamation property
\cite[\S7.1, Example~1]{MR1221741}, but
which does not have group-theoretic trivial definable closure.

Finally, Theorem~\ref{KMConditions} has the following consequence, which
(for relational languages) extends Lemma~\ref{alephnought}
on the relationship between trivial definable closure and
duplication of quantifier-free types.

\begin{corollary}
\label{partialconverse}
Let $L$ be a countable relational language, and let $T$ be a countable pithy $\Pi_2$ theory of $\Lwow(L)$ that has a unique countable model $\cM$ (up to isomorphism).  Then $\cM$ has trivial definable closure if and only if $T$ has duplication of quantifier-free types.
\end{corollary}
\begin{proof}
From Lemma~\ref{alephnought} it is immediate 
that if $\cM$ has trivial definable closure then $T$ must have duplication of quantifier-free types.

For the other direction, suppose that $T$ has duplication of
quantifier-free types.
By Theorem~\ref{ModelsWitnessingT} there is a samplable Borel
$L$-structure strongly witnessing $T$, which by
Corollary~\ref{stronglywitnessingcorollary} induces an invariant
probability measure concentrated on the set of models of $T$ in $\Model$. Hence
there is an invariant measure concentrated on $\cM$. But Theorem~\ref{KMConditions} then implies that $\cM$ has trivial definable closure.
\end{proof}


\section{Structures admitting invariant measures}\label{examples-nonexamples}

We now consider several important classes of structures, and
examine whether or not the structures in these classes
admit invariant measures.
In \S\ref{equivalence-relation}, we show how any countable infinite structure
is a quotient of one with trivial definable closure, and of one without.  We use this fact to construct countable structures of arbitrary Scott rank that have trivial definable closure and hence admit invariant measures, and to construct ones that do not.
In \S\ref{combexamples} we apply our
main results, Theorem~\ref{maintheorem} and Corollary~\ref{maincorollary},
to examine certain well-known 
countable infinite structures, and ask whether or not they admit
invariant measures.
We make
use of existing classifications to provide complete lists of countable
infinite 
ultrahomogeneous 
partial orders,
permutations,
directed graphs,
and graphs, 
for which such invariant measures exist.  

\subsection{Structures with an equivalence relation}\label{equivalence-relation}
\ \\ \indent
Suppose  we are given a countable infinite structure in a countable relational language
with a binary relation symbol. Further, suppose that this symbol is interpreted as an equivalence
relation such that every equivalence class has at least two elements.
Consider
the quotient map on the underlying set that is induced by this equivalence relation.
In the case where this quotient map
\emph{respects} the remaining
relations (in a sense that we will make precise), we can
characterize when the original structure does or does not have trivial definable closure.
On the other hand, starting with an arbitrary countable infinite structure in a countable relational
language,
we can ``blow up'' each element into an
equivalence class, and characterize when the resulting structure
has trivial definable closure.

We will thereby see, in 
Corollary~\ref{measure-split-equivalence-class},
that every countable structure in a countable relational
language is the quotient of one with trivial definable closure, and
of one without.
We then apply this result to further yield 
Corollary~\ref{scott-rank-cor}, and
obtain structures of arbitrary Scott rank that
admit invariant measures, as well as ones that do not.

We begin by describing what it means for an equivalence relation
to respect the remaining relations in the language.
\begin{definition}
For a relational language $L$,
define $L^+ \defas L\cup \{\equiv\}$, where
$\equiv$ is a new binary relation symbol. Let $\cN$ be an
$L^+$-structure. We say that $\equiv$ \defn{respects $L$ in $\cN$}
if for each $k$-ary (non-equality) relation symbol $R \in L$,
\[
\cN \models (\forall x_1, \dots, x_k, y_1, \dots, y_k)
\bigwedge_{1 \leq i \leq k} (x_i \equiv y_i) \rightarrow \bigl(R(x_1, \dots, x_k)
\leftrightarrow R(y_1, \dots, y_k)\bigr).
\]
Such a relation $\equiv$ is often referred to as an \emph{equality}.
\end{definition}
When $\equiv$ respects $L$ in $\cN$, the
structure $\cN$ cannot 
``$L$-distinguish'' between $\equiv$-equivalent elements. In particular, there is a 
quotient structure
induced by the $\equiv$ relation on the underlying set.

For a countable infinite $L^+$-structure $\cN$ in which $\equiv$ respects $L$, and
where every $\equiv$-equivalence class has at least two elements, 
the size of the $\equiv$-equivalence
classes completely determines  whether or not
$\cN$ admits an invariant measure.
\begin{lemma}
\label{measure-and-equivalence-relations}
Suppose 
$\cN$ is a countable infinite $L^+$-structure
such that
$\equiv$ respects $L$ in $\cN$ and such that no $\equiv$-equivalence class has
only one element. The following are equivalent:
\begin{itemize}
\item[(1)] Every $\equiv$-equivalence class of $\cN$ has infinitely many elements.
\item[(2)] There is an invariant measure on $\Str_{L^+}$ that is concentrated on the isomorphism class of $\cN$.
\end{itemize}
\end{lemma}
\begin{proof}

Assume that $(1)$ holds.
Whenever $c,c'\in\cN$
are such that $\cN \models (c \equiv c')$,
define $g_{c, c'} \colon \cN \to \cN$ to be
the map that interchanges $c$ and $c'$ but is constant on all other elements of $\cN$.
Since $\equiv$ respects $L$, the map $g_{c,c'}$ is an automorphism of $\cN$.

Suppose, for a contradiction, that there  are $\aa, b\in \cN$ such that
$b \in \dcl_{\cN}(\aa) -\aa$.  Each $\equiv$-equivalence class
has infinitely many elements, and so there must be some $b'\in\cN$
satisfying $b'\not \in \aa b$ and $\cN\models (b \equiv b')$.  Now, $g_{b,
b'}$ fixes $\aa$ pointwise by construction.  Because $b\in\dcl_{\cN}(\aa)$,
the map $g_{b,b'}$ also fixes $b$.  Hence $b = g_{b,b'}(b) = b'$, a
contradiction.  Therefore $\cN$ has trivial definable closure, and so by
Theorem~\ref{maintheorem}, $\cN$ admits an invariant measure.

For the converse, assume that $(1)$ fails.  Let $A\subseteq\cN$ be a
finite $\equiv$-equivalence class.
By hypothesis, 
$A$ has at least two elements. Hence for each $a \in A$, the set
$A - \{a\}$ is nonempty, and so $a \in \dcl_{\cN}(A - \{a\})$. Therefore
$\cN$ has nontrivial definable closure,
and so $(2)$ fails by 
Theorem~\ref{maintheorem}.
\end{proof}

From Lemma \ref{measure-and-equivalence-relations} we can see that, in a sense, every 
countable infinite $L$-structure is ``close to'' one that admits an invariant measure, and
also to infinitely many that do not.
Specifically, if we take a {countable infinite $L$-structure} and ``blow up'' every element into
$n$-many elements, where $n$ is a cardinal satisfying $1 < n \leq \aleph_0$, then the resulting structure admits
an invariant measure if and only if $n = \aleph_0$.

\begin{definition}
Let $L$ be a relational language, and let 
$\cM$ be a countable infinite $L$-structure with underlying set $M$.
Suppose $n$ is a cardinal satisfying $1 \le n \leq \aleph_0$.  Define $\cM^+_n$ to be the $L^+$-structure with underlying set $M \times n$ such that 
\[
\cM^+_n \models (a, j) \equiv (a',  j')
\qquad \text{if and only if} \qquad 
a  = a',
\]
for every $(a,j), (a', j') \in M\times n$, and
\[
\cM^+_n\models R\bigl((a_1, j_1), \ldots, (a_k, j_k)\bigr)
\qquad \text{if and only if} \qquad 
\cM \models R(a_1, \ldots  a_k),
\]
for every relation $R\in L$ and every $(a_1, j_1), \ldots, (a_k, j_k)\in M\times n$, where $k$ is the arity of $R$.
\end{definition}

In the case when $\cM$ is a graph, this construction is known as the \emph{lexicographic product} of $\cM$ with the empty graph on $n$ vertices.

Note that $\equiv$ is an equivalence relation on $\cM^+_n$ that respects
$L$ in $\cM^+_n$.  Hence we may take the quotient of $\cM^+_n$ by
$\equiv$ to obtain a structure isomorphic to $\cM$.
Moreover, every $\equiv$-equivalence class of $\cM^+_n$ has $n$-many elements.

As an immediate corollary of Lemma \ref{measure-and-equivalence-relations} we have the following.

\begin{corollary}
\label{measure-split-equivalence-class}
Let $L$ be a relational language, 
let 
$\cM$ be a countable infinite $L$-structure,
and
let $n$ be a cardinal such that $1 < n \leq \aleph_0$.
Then
$\cM^+_{\aleph_0}$ admits an invariant measure, while
for $1 < n < \aleph_0$, the structure $\cM^+_n$ does not
admit an invariant measure.
\end{corollary}

Note that this shows that every countable structure in a countable
language is interpretable (see, e.g., \cite[Definition~1.3.9]{MR1924282}) in a structure that admits an
invariant measure, as well as in a structure that does not.

The \emph{Scott rank} of a structure provides a measure of the complexity of the Scott
sentence of the structure. (For details, see \cite{MR2289895}.)
Corollary \ref{measure-split-equivalence-class} 
provides a method by which to build countable structures of
arbitrary Scott rank that admit invariant measures, as well as ones
that do not.
\begin{corollary}
\label{scott-rank-cor}
Let $\alpha$ be an arbitrary countably infinite ordinal.  Define $\T^\alpha$ to be a countable linear order isomorphic to the well-order $(\alpha,\in)$ of height $\alpha$.  Then the structure $(\T^{\alpha})^+_{\aleph_0}$ has Scott rank $\alpha$ and admits an invariant measure, whereas for $1 \leq n <\aleph_0$, the structure $(\T^{\alpha})^+_{n}$ has Scott rank $\alpha$ and does not admit an invariant measure.
\end{corollary}
\begin{proof}
	For $1 \le n \le \aleph_0$,
the structure
$(\T^{\alpha})^+_{n}$ has Scott rank $\alpha$, as can be seen by a simple back-and-forth argument with $(\T^{\beta})^+_{n}$ for $\beta < \alpha$.  For $1< n \leq \aleph_0$, the result follows by Corollary \ref{measure-split-equivalence-class}. When $n=1$, the result follows from the fact that the 
least element of
$(\T^{\alpha})^+_{1}$ is in the definable closure of the empty set.
\end{proof}

\subsection{Classifications and other examples}\label{combexamples}
\ \\ \indent
Here we examine certain well-known countable infinite structures, 
and note whether or not they admit invariant measures.  
In some cases, such as the countable universal ultrahomogeneous
partial order, our results provide the first demonstration that the
structure admits an invariant measure.
In several instances, the existence of invariant measures was known
previously, though our results provide a simple way to check this.
For example, 
it has been known nearly since its initial construction
that the Rado graph $\Rado$ admits an invariant measure,
and Petrov and Vershik \cite{MR2724668}   have more recently constructed
invariant measures concentrated on the Henson graph $\Henson$ and on the
other countable universal ultrahomogeneous $K_n$-free graphs.

Our results may be used to determine whether a particular structure admits an invariant measure
either by checking directly whether it has trivial definable closure and applying Theorem~\ref{maintheorem}, or, in the case of an ultrahomogeneous structure in a relational language, by  determining whether its age has
strong amalgamation and applying Corollary~\ref{maincorollary}.
It will be convenient sometimes to use the fact that a structure has trivial definable closure if and only if it has
trivial algebraic closure, as mentioned in \S\ref{newstrongamalgamation}.

In the examples below, all graphs, directed graphs, and partial orders 
are considered to be structures in a language with a single binary relation symbol.


\subsubsection{Countable infinite ultrahomogeneous partial orders}
\label{posets}
\ \\
\indent
These have been classified by Schmerl \cite{MR544855} as follows.
\begin{enumerate}
	\item[\emph{(a)}]
The rationals, $(\Rationals, <)$.
	\item[\emph{(b)}]
The countable universal ultrahomogeneous partial order.
	\item[\emph{(c)}]
The countable infinite antichain.
	\item[\emph{(d)}]
The antichain of $n$
	copies of $\Rationals$ ($1 < n \le \omega$).
	\item[\emph{(e)}]
The $\Rationals$-chain of antichains, each of size $n$ ($1\le n< \omega$).
	\item[\emph{(f)}]
The $\Rationals$-chain of antichains, each of size $\omega$.
\end{enumerate}
All but \emph{(e)} admit invariant measures: Their amalgamation problems can be solved by
taking the transitive closure and, when needed, linearizing, and so their ages exhibit
strong amalgamation. Example \emph{(e)} clearly has nontrivial algebraic
closure, and so does not admit an invariant measure.


\subsubsection{Countable infinite ultrahomogeneous permutations}
\ \\
\indent
Finite permutations have a standard interpretation as structures in a language with 
two binary relation symbols \cite{MR2028272} (see \cite[\S4.1]{cherlin} for
a discussion).
A permutation $\sigma$ on $\{1, \ldots, n\}$ can be viewed as two linear 
orders,
$<$ and $\lhd$, on $\{1, \ldots, n\}$, where $<$ is the usual order, and $\lhd$
is the permuted order, i.e., $\sigma(a) \lhd \sigma(b)$ if and only if $a < b$.
One may extend this perspective on permutations to the infinite case, and consider structures 
that consist of a
single infinite set endowed with two linear orders. Such structures describe relative
finite rearrangements without completely determining a permutation on the infinite set. 
The countable 
infinite
ultrahomogeneous permutations, so defined,
have been classified by Cameron \cite{MR2028272} as follows.

\begin{enumerate}
	\item[\emph{(a)}]
The rationals, 
i.e., where each linear order has order type $\Rationals$ and they are
equal to each other.
	\item[\emph{(b)}]
The reversed rationals, 
i.e., where each linear order has order type $\Rationals$ and the second
is the reverse of the first.
	\item[\emph{(c)}]
Rational blocks of reversed rationals, i.e., where each linear order is
the lexicographic product of $\Rationals$ with itself, and the second
order is the reverse of the first \emph{within} each block.
	\item[\emph{(d)}]
Reversed rational blocks of rationals, i.e., where each linear order is
the lexicographic product of $\Rationals$ with itself, and the second
order is the reverse of the first \emph{between} the blocks.
	\item[\emph{(e)}]
The countable universal ultrahomogeneous permutation.
\end{enumerate}

All five have trivial definable closure and hence admit invariant measures.

\subsubsection{Countable infinite ultrahomogeneous tournaments}
\label{tournaments}
\ \\
\indent
A \emph{tournament} is a structure consisting of a single irreflexive,
binary  relation, $\to$, such that for each pair $a,b$ of distinct
vertices, either $a \to b$ or $b\to a$, but not both. For example, any
linear order is a tournament.
The countable infinite ultrahomogeneous tournaments have been classified by
Lachlan \cite{MR743728} as follows.

\begin{enumerate}
	\item[\emph{(a)}]
The rationals, $(\Rationals, <)$.
	\item[\emph{(b)}]
The countable universal ultrahomogeneous tournament, $T^\infty$.
	\item[\emph{(c)}]
The \emph{circular tournament} $S(2)$, also known as the \emph{local order}, 
which consists of a countable 
dense subset of a circle where no two points are antipodal, with $x \to
y$ if and only if
the angle of $xOy$ is less than $\pi$, where $O$ is the center of the
circle.
\end{enumerate}
The ages of all three exhibit strong amalgamation (see
\cite[\S2.1]{MR1434988}).


\subsubsection{Countable infinite ultrahomogeneous directed graphs}
\ \\
\indent
A \emph{directed graph} is a structure consisting of a single irreflexive,
binary  relation, $\to$, that is asymmetric, i.e., such that for each pair $a,b$ of distinct
vertices, $a \to b$ and $b\to a$ do not both hold. The countable infinite
ultrahomogeneous directed graphs have been classified by Cherlin
\cite{MR1434988}
(see also \cite{MR895639} for the imprimitive case). 
Macpherson \cite{MR2800979} describes the classification
as follows (with some overlap between classes).
\begin{enumerate}
	\item[\emph{(a)}]
The countable infinite ultrahomogeneous partial orders.
	\item[\emph{(b)}]
The countable infinite ultrahomogeneous tournaments.
	\item[\emph{(c)}]
Henson's countable infinite ultrahomogeneous directed graphs with
forbidden sets of tournaments.
	\item[\emph{(d)}]
The countable infinite ultrahomogeneous directed graph omitting $I_n$,
the edgeless directed graph on $n$ vertices ($1<n<\omega$).
	\item[\emph{(e)}]
Four classes of directed graphs that are imprimitive, i.e., for which there is a nontrivial
equivalence relation definable without parameters.
	\item[\emph{(f)}]
Two exceptional directed graphs: 
a shuffled $3$-tournament $S(3)$, defined analogously to the local order
(defined above in \ref{tournaments}(c))
with angle
$2\pi/3$, and 
the \emph{dense local partial order} $\mathcal{P}(3)$, a modification of
the countable universal ultrahomogeneous partial order.
\end{enumerate}

The structures in 
\emph{(a)}
and
\emph{(b)} are discussed above, in  \S\ref{posets} and \S\ref{tournaments},
respectively.

Henson \cite{MR0321727} described the class 
\emph{(c)}
of $2^{\aleph_0}$-many
nonisomorphic countable infinite ultrahomogeneous directed graphs with
forbidden sets of tournaments. 
The age of each has \emph{free} amalgamation,
i.e., its amalgamation problem can be solved by taking the disjoint union
over the common substructure and adding no new relations.   Free
amalgamation implies strong amalgamation; 
hence on Henson's ultrahomogeneous directed graphs there are invariant
measures.

The ages of the structures in \emph{(d)} have strong amalgamation.

The first imprimitive class in \emph{(e)} consists of the wreath
products $T[I_n]$ and $I_n[T]$ where $T$ is a countable infinite
ultrahomogeneous tournament (as discussed above in \S\ref{tournaments})
 and $1 < n < \omega$.
Each $T[I_n]$ has nontrivial definable closure because there is a definable
equivalence relation, each class of which has $n$ elements.
Each $I_n[T]$ has trivial definable closure because it is the disjoint
union of copies of an infinite tournament that
has strong amalgamation.

The second imprimitive class
in \emph{(e)}
consists of $\widehat{\Rationals}$ and
$\widehat{T^\infty}$, modifications of the rationals and the countable
universal ultrahomogeneous tournament, respectively, in which the algebraic closure of
each point has size $2$, namely itself and the unique other point to which it
is not related. Hence neither directed graph has trivial definable closure.

The third imprimitive class 
in \emph{(e)}
consists of directed graphs $n * I_\infty$, for $1 < n \le \omega$,
which are universal subject to the constraint that non-relatedness is
an equivalence relation with $n$ classes. 
All such
directed graphs have trivial definable closure.

The fourth imprimitive class 
in \emph{(e)}
consists of a \emph{semigeneric} variant
of $\omega * I_\infty$ with a parity constraint, which also has trivial definable closure.

The ages of $S(3)$ and $\mathcal{P}(3)$ exhibit strong amalgamation.

\subsubsection{Countable infinite ultrahomogeneous graphs}
\label{LachlanWoodrow}
\ \\
\indent
These have been classified by Lachlan and Woodrow \cite{MR583847} as follows.
	\begin{enumerate}
	\item[\emph{(a)}]
The Rado graph $\Rado$. 
	\item[\emph{(b)}]
The Henson graph $\Henson$
and the other
countable universal ultrahomogeneous  $K_n$-free graphs $(n>3)$, and their 			complements.
	\item[\emph{(c)}]
Finite or countably infinite union of $K_\omega$, and their complements.
	\item[\emph{(d)}]
Countably infinite union of $K_n$ (for $1< n<\omega$), and their complements.
\end{enumerate}
The ages of the structures in \emph{(a)} through \emph{(c)} all have strong
amalgamation; in fact, for the Rado graph, Henson's $\Henson$ and other
$K_n$-free graphs, and the complement of $K_\omega$, the amalgamation
is free.
Hence the structures in \emph{(a)} through \emph{(c)} all admit invariant
measures.  The structures in \emph{(d)} clearly have nontrivial algebraic
closure, and so do not admit invariant measures.


\subsubsection{Countable universal $C$-free graphs}  
\ \\
\indent
Let $C$ be a finite set of  finite connected graphs. 
A graph $\G$ is said to be \emph{$C$-free}, or to \emph{forbid $C$}, when no member of $C$ is isomorphic to a
(graph-theoretic) subgraph of $\G$, i.e., when no member of $C$ embeds as a weak
substructure of  $\G$.
A countable infinite $C$-free graph  $\G$ is said to be \emph{universal}
when every
countable $C$-free graph is isomorphic to an induced subgraph of $\G$, i.e., embeds as a
substructure of $\G$. When there is a universal such graph, there is one (up to isomorphism) that is distinguished by being \emph{existentially complete}.

Only a limited number of examples are known of finite sets $C$ of finite
connected graphs for which
a countable universal $C$-free graph exists (see the introduction to
\cite{MR1683298} for a discussion).
The best-known are when $C = \{K_n\}$, for \hbox{$n\ge 3$;} Henson's
countable universal ultrahomogeneous $K_n$-free graph is universal for
countable graphs that forbid $\{K_n\}$.
We consider two other families here.

\emph{(a)} 
The set $C$ is homomorphism-closed, i.e., closed under maps that preserve edges but not
necessarily non-edges. 
For example, take $C$ to be the set of cycles of all odd
lengths up to a fixed $2n+1$.
Cherlin, Shelah, and Shi
\cite[Theorem~4]{MR1683298} 
have shown
that
for a homomorphism-closed set $C$,
an existentially complete countable universal
$C$-free graph exists and 
has trivial algebraic closure. Hence these graphs admit
invariant measures.
Such graphs 
have also been considered in \cite{2009arXiv0909.4939H}.

\emph{(b)} The singleton set $C = \{K_m \plusdot K_n\}$ for some $m,n>2$, where
 $K_m \plusdot K_n$ is the graph on $m+n-1$ vertices consisting of complete
graphs $K_m$ and $K_n$ joined at
a single vertex. For example, $K_3\plusdot K_3$ is the
so-called \emph{bowtie}.
An existentially complete countable universal $(K_m \plusdot K_n)$-free graph exists if and only if $\min(m,n) = 3$ or $4$,  or $\min(m,n)
=5$ but
$m\neq n$
(\cite{MR1675931}, \cite{MR1683298},  and \cite{MR2342786}). 
Any such 
graph 
has nontrivial
algebraic closure because, by existential completeness,
it must contain a copy $\cK$ of $K_{m+n-2}$, but for any 
vertex $v\in\cK$, the algebraic closure of $\{v\}$ in the graph is all of
$\cK$.


\subsubsection{Trees and connected graphs with finite cut-sets}
\ \\
\indent
A tree is an acyclic connected graph. 
No tree can have trivial algebraic closure because there exists a unique finite
path between any two distinct vertices of the tree.
Similarly, no connected graph with a cut-vertex 
(a vertex whose removal disconnects the graph)
can have trivial algebraic closure.
More generally, if a connected graph contains a finite cut-set
(a finite set whose removal disconnects the graph), then it cannot have trivial
algebraic closure.


\subsubsection{Rational Urysohn space}
\ \\
\indent
A rational metric space is a metric space all of whose distances are
rational.
The class of all finite rational metric spaces, considered
in the language with one relation symbol for each rational distance, 
is a \Fr\ class. 
Its \Fr\ limit is known as the \emph{rational Urysohn space}, denoted
$\QU$ (for details see
\cite{MR2258622}). 
The completion of $\QU$ is the Urysohn space,
	the universal ultrahomogeneous
	complete separable metric space.

	The space $\QU$ admits an invariant measure,
	as can be seen from our results, 
since the
class of finite rational metric spaces has strong amalgamation.
Vershik, in \cite{MR2006015} and \cite{MR2086637},
has earlier constructed invariant measures concentrated on 
a collection of countable metric spaces whose completions are also Urysohn space.
For a construction of several related invariant measures,
see \cite{AFNP}.


\section{Applications and further observations}
\label{more-general}

We conclude the paper with some observations and applications of our
results.
We describe, in \S\ref{graphlimits},
some of the theory of dense graph limits and its connections to our
setting.

Our main theorem, Theorem~\ref{maintheorem}, completely characterizes
those single orbits of $\sym$ on which an
invariant measure can be concentrated.
In \S\ref{multipleisom},
we ask
which other Borel subsets of $\Model$, consisting of multiple orbits, are such that some invariant measure is
concentrated on them, and we make some observations based on our machinery.

Finally, in \S\ref{scottcontinuum}, 
we note a corollary of our result for sentences of $\Lwow$ that have exactly one model (countable or otherwise).


\subsection{Invariant measures and dense graph limits}
\label{graphlimits}
\ \\
\indent
As remarked in the introduction, 
our constructions in the case of graphs can be viewed within the
framework of the
theory of dense graph limits.
Here we describe this connection and some of its consequences.

\subsubsection{Invariant measures via graphons and $W$-random graphs}
\ \\ \indent
We now describe how invariant
measures arise in the context of
dense graph limits.
We begin with some definitions from \cite{MR2274085};
for more details, see also \cite{MR2463439}, and \cite{MR3012035}.

A \defn{graphon} is defined 
to be a symmetric measurable  function $W\colon [0,1]^2\to[0,1]$. 
In what follows, we will take all graphons to be Borel measurable.
Let $L_G$ be the language of graphs, i.e., 
a
language consisting of a
single binary relation symbol $E$, representing the edges.
Let $T_G$ be the theory in the language $L_G$ that says that $E$ is symmetric and irreflexive.
A graph may be considered to be an $L_G$-structure that satisfies $T_G$.
An invariant measure on graphs is then precisely
an invariant measure on $\Str_{L_G}$ that is
concentrated on the set of models of $T_G$ in $\Str_{L_G}$.

Given a graphon $W$, the \defn{$W$-random graph} $\GG(\Nats, W)$
can be thought of as a random element of $\Str_{L_G}$, defined as follows.
Let $\{X_k\}_{k\in\Naturals}$ be an independent sequence of random variables uniformly
distributed on the unit interval.
Then for $i,j\in\Naturals$ with $i<j$, let
$E(i,j)$ hold with independent probability $W(X_i,X_j)$; for each $i$, require that
$E(i, i)$ not hold; and for each $i>j$, let $E(i, j)$ hold if and only if $E(j, i)$ does.
For example, when $W$ is a constant function $p$ where $0<p<1$, then
$\GG(\Nats, W)$ is essentially the Erd\H{o}s-R\'{e}nyi graph $\GG(\Nats, p)$, described in \S\ref{background}.
Notice that for any graphon $W$, the distribution of $\GG(\Nats, W)$ is an invariant measure on graphs.

Not only is the distribution of $\GG(\Nats, W)$ invariant for an arbitrary graphon $W$, 
but so are the mixtures, i.e., convex combinations, of 
such distributions.
Conversely,
Aldous \cite{MR637937} and Hoover \cite{Hoover79}
showed, in the context of exchangeable random arrays, that \emph{every} invariant measure on graphs is such a mixture,
thereby completely characterizing the invariant measures on graphs.
This characterization has also
arisen in the theory of dense graph limits; for details see \cite{MR2463439} and \cite{MR2426176}.

An analogous theory to that of graphons has been developed for other
combinatorial structures such as partial orders
\cite{DBLP:journals/combinatorica/Janson11} and permutations
\cite{MR2995721}. The standard recipe described in \cite{MR2426176} extends this machinery to the general case of countable relational languages of bounded arity.  When $L$ has bounded arity, our notion of Borel $L$-structure, from \S\ref{BorelLStructuresSection}, can be viewed as a specialization of certain structures that occur in the standard recipe.  In particular, any Borel $L_G$-structure that is a model of $T_G$ corresponds to a graphon, as we will now see.
Recall that because $L_G$ is relational, every Borel $L_G$-structure is samplable.

\subsubsection{Borel $L_G$-structures and random-free graphons}
\ \\ \indent
Borel $L_G$-structures  that are models of $T_G$ (i.e., graphs) are closely related to a
particular class of graphons. 
Here we describe this relationship
and use  it
to deduce a corollary about 
$W$-random graphs whose distributions are concentrated on single
countable graphs.

A graphon $W$ is said to be \defn{random-free} 
\cite[\S10]{MR3043217} 
if for a.e.\ $(x,y)\in[0,1]^2$ we have
$W(x,y) \in \{0,1\}$.
(See also the \emph{simple arrays} of \cite{MR1702867} and $0$\,--\,$1$ valued graphons in
\cite{MR2815610}.)
When $W$ is random-free, the $W$-random graph process amounts,
in the language of \cite{MR2724668}, to ``randomization in vertices'' but not ``randomization in edges''.

We now describe a correspondence between Borel $L_G$-structures 
satisfying $T_G$ and
random-free graphons.
Let $\alpha$ be an arbitrary Borel measurable bijection from the open interval $(0,1)$ to $\Reals$, 
and let $m_{\alpha}$ be the distribution of $\alpha(U)$ where $U$ is uniformly  distributed on $[0,1]$.
Given a
Borel $L_G$-structure $\PP$ that satisfies $T_G$,  
define
the random-free graphon 
$W_{\PP}$  as follows.
For $(x,y) \in (0,1)^2$ let
\[
W_{\PP}(x,y) = 1 \qquad \text{if and only if} \qquad \PP \models
E\bigl(\alpha(x), \alpha(y)\bigr),
\]
and for $(x,y)$ on the boundary of $[0,1]^2$
let $W_{\PP}(x,y) = 0$. The
distribution of  $\GG(\Nats, W_{\PP})$ is precisely $\mu_{(\PP,
m_{\alpha})}$, as defined in Definition~\ref{mupm}. 
Conversely, given a graphon $W$ that is Borel and random-free, one can easily build a Borel
$L_G$-structure $\PP_W$ satisfying $T_G$
such that the distribution of $\GG(\Nats, W)$ is $\mu_{(\PP_W, m_{\alpha})}$.

By Corollary~\ref{reduct}, if a countable graph admits an invariant measure,
then it admits one of the form $\mu_{(\PP, m)}$, where $\PP$ is a 
Borel ${L_G}$-structure.
In particular, the
corresponding random-free graphon $W_{\PP}$ is such
that the distribution of $\GG(\Nats, W_{\PP})$ is an invariant measure concentrated on 
the given graph.
This leads to the following corollary.

\begin{corollary}
\label{randomfreecor}
Let $\cM$ be
a countable infinite graph. Suppose there is some graphon $W$ such that
the distribution of $\GG(\Nats, W)$ is concentrated on $\cM$. 
Then there is a random-free graphon $W'$ such that the distribution of
$\GG(\Nats, W')$ is also concentrated on $\cM$.
\end{corollary}
\begin{proof}
The distribution of $\GG(\Nats, W)$ is an invariant measure concentrated on
$\cM$.
Therefore by Theorem~\ref{maintheorem}, the graph $\cM$ must have trivial definable closure. 
By Corollary~\ref{reduct}, there is a Borel $L_G$-structure $\PP$ such that
$\mu_{(\PP,m)}$ is concentrated on $\cM$ whenever $m$ is a continuous nondegenerate probability
measure on $\Reals$.
As above, let
$\alpha \colon (0,1) \to \Reals$ be a Borel bijection 
and let $W_{\PP}$ be the
random-free graphon induced by the given correspondence.
Then the distribution 
of $\GG(\Nats, W_\PP)$  is
$\mu_{(\PP, m_\alpha)}$, and hence
is also concentrated on $\cM$.
\end{proof}

In fact, for an arbitrary
countable relational language $L$,
our procedure for sampling from a Borel $L$-structure
essentially arises in
\cite{MR2426176} 
as a standard recipe in
which all but the first ``ingredient'' are deterministic maps.
In this setting, one can prove an analogue of  Corollary~\ref{randomfreecor}
for arbitrary countable infinite $L$-structures.

	The best-known graphons $W$  for which
	$\GG(\Nats, W)$ is isomorphic to the Rado graph are the constant functions
	$W \equiv p$ 
	for $0<p<1$, i.e.,
	those given by
the Erd{\H{o}}s--R\'enyi construction. However, these are not the only such graphons.
	Petrov and Vershik \cite{MR2724668} were the first to describe invariant measures concentrated on the
	Rado graph 
	that correspond to random-free graphons.  Figure~\ref{radoon}  is a visualization of a random-free graphon $W$,
built essentially by the methods of \cite{MR2724668} and the present paper,
	for which
$\GG(\Nats, W)$ is a.s.\ isomorphic to the Rado graph, 

\begin{figure}[h]
{
\hspace{-20pt}\!\!
\begin{center}
\includegraphics[width=0.50\linewidth]{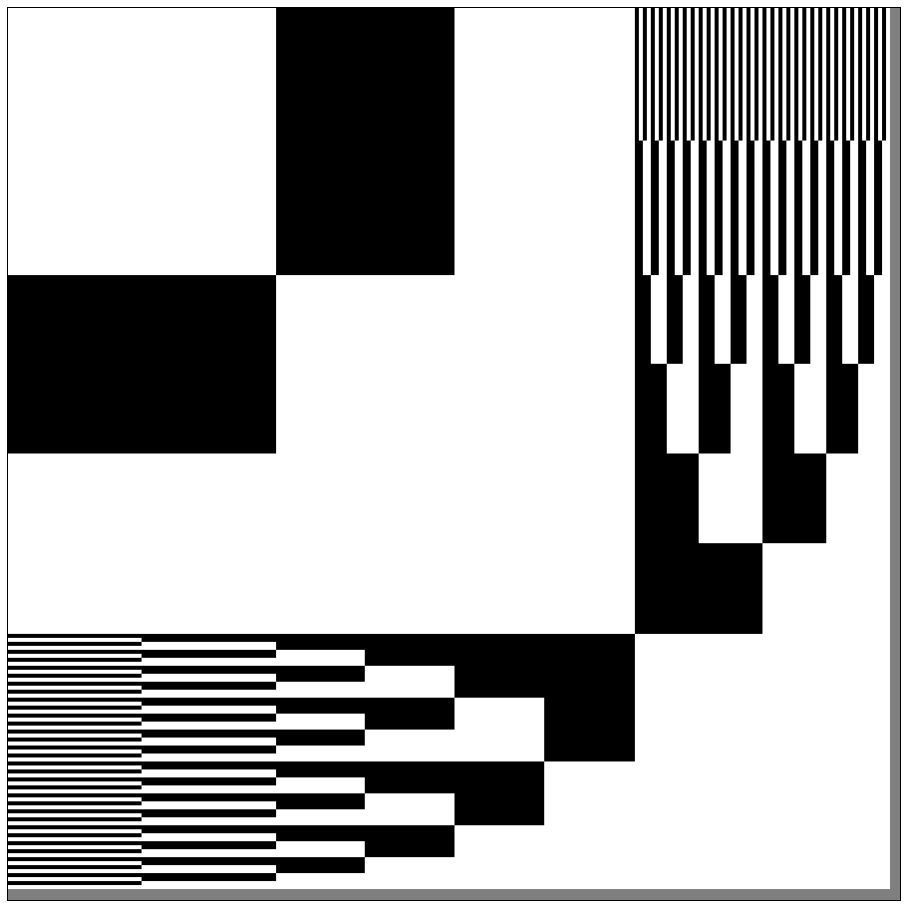}
\caption{An illustration of a random-free graphon $W$ such that $\GG(\Nats, W)$ is a.s.\ isomorphic to the Rado graph. (The thin grey strips on the right and bottom represent regions not drawn in detail --- not values of the graphon between $0$ and $1$.)}
\label{radoon}
\end{center}
}
\end{figure}

\subsection{Multiple isomorphism classes}
\label{multipleisom}
\ \\ \indent
In this paper, we have focused on the problem of identifying those countable infinite $L$-structures $\cM$ such that some invariant measure is concentrated on the isomorphism class of $\cM$ in $\Model$, i.e., on the orbit under the logic action of any structure in $\Model$ isomorphic to $\cM$.  But it is natural to investigate those larger subsets of $\Model$, consisting of the union of multiple orbits, on which an invariant measure may be concentrated.  For example, Austin \cite[Question~3.27]{MR2426176} asks for a characterization of first-order theories $T$ such that any invariant measure concentrated on the set of
models of $T$ in $\Model$ must come from a standard recipe having a property akin to being random-free.


There are clearly invariant measures on $\Model$ that are not concentrated on any single structure, as can be seen by taking mixtures of invariant measures concentrated on different structures. But if there is a countable set of structures on which such an invariant measure is concentrated, then we can see by conditioning
that there must be 
some
invariant measure concentrated on one of these
structures.

\begin{lemma}
Let $L$ be a countable language, and let $T$ be a theory of $\Lwow(L)$ that
has at most countably many countable infinite models (up to isomorphism).  Suppose $\mu_T$ is an invariant measure on $\Model$ that is concentrated on the set of models of $T$ in $\Model$.  Then there is a countable model $\cM$ of $T$ such that some invariant measure $\mu_{\cM}$ is concentrated on the isomorphism class of $\cM$ in $\Model$.
\end{lemma}

\begin{proof}
Because $\mu_T$ is countably additive and $T$ has only countably many
countable infinite models, there must be some countable infinite structure $\cM \models T$ such that its
isomorphism class $\widetilde{\cM} \defas \{\cN\in \Model\st \cN\cong \cM\}$ has positive $\mu_T$-measure.  Recall that $\widetilde{\cM}$ is a Borel set.  Let $\mu_\cM$ be $\mu_T$ conditioned on this positive measure set, i.e., 
\[
\mu_{\cM}(A) \defas \mu_T\bigl(A\,|\,\widetilde{\cM}\bigr)= \mu_T\bigl(A \cap
\widetilde{\cM}\bigr)/ \mu_T\bigl(\widetilde{\cM}\bigr)
\]
for every Borel set $A\subseteq\Model$. Then $\mu_\cM$ is a probability
measure on $\Model$
concentrated on the isomorphism class of $\cM$.

Moreover,
$\mu_\cM$
is invariant, as we now show. Suppose $g \in \sym$, and
let $A$ be an arbitrary Borel subset of $\Model$.
Because
$\widetilde{\cM}$ is an $\sym$-invariant subset of $\Model$, we have
\[
\mu_T\bigl(g(A) \cap \widetilde{\cM}\bigr)
= \mu_T\bigl(g(A) \cap g(\widetilde{\cM})\bigr),
\]
and because $\mu_T$ is an invariant measure, we have
\[
\mu_T\bigl(g(A \cap \widetilde{\cM})\bigr)
= \mu_T\bigl(A \cap \widetilde{\cM}\bigr).
\]
Since $g(A) \cap g\bigl(\widetilde{\cM}\bigr) = g\bigl(A \cap \widetilde{\cM}\bigr)$, we
have
$ \mu_{\cM}(g(A)) = \mu_{\cM}(A)$, as desired.
\end{proof}


One may ask, more specifically, given a samplable Borel $L$-structure $\PP$ and a continuous
nondegenerate probability measure $m$ on $\Reals$, the
minimum number of isomorphism classes on whose union the measure $\mu_{(\PP,m)}$ is concentrated.
When $\PP$
strongly witnesses a theory $T$ of $\Lwow(L)$ 
having just one countable infinite model up to isomorphism,
then there is just one isomorphism class by design.
However, if $\PP$ strongly witnesses a pithy $\Pi_2$ theory $T$ of $\Lwow(L)$ 
that 
has 
nonisomorphic
countable infinite models, then
the situation is more complicated. 
In this case, 
still
$\PP \models T$ 
by
Lemma~\ref{SubmodelsSatisfyingT}, but the induced invariant measure
might be concentrated on a union of multiple 
isomorphism classes of
models of $T$, but not on any single such class.

However,
as we state in Corollary~\ref{ergodic-cor},
this is not possible if
the measure is concentrated on a \emph{countable} union
of isomorphism classes.
By countable additivity,
any invariant measure 
on $\Model$
that is concentrated on a union of countably
many isomorphism classes, but not on a single class, 
must be
non-ergodic.
But every 
measure of the form $\mu_{(\PP,m)}$ is ergodic, as we now explain.

	The ergodic invariant measures on graphs
	are precisely those induced by sampling from a single graphon, rather than a mixture of such
(see  \cite[Corollary~5.4]{MR2463439} and \cite[Proposition~3.6]{JGT:JGT20611}).
Aldous had earlier characterized 
	the ergodic invariant measures on hypergraphs in a similar way (see \cite[Proposition~3.3]{MR637937} or \cite[Lemma~7.35]{MR2161313}). This characterization has a generalization to countable infinite languages, e.g., via the setting of Kallenberg's extension of the Aldous--Hoover theorem (\cite[Lemma~7.22]{MR2161313} and \cite[Lemma~7.28]{MR2161313}).
In particular, it can be shown that measures of the form $\mu_{(\PP,m)}$ are ergodic,
and so the following corollary holds.

\begin{corollary}
	\label{ergodic-cor}
Let
$\PP$ be 
a samplable
Borel $L$-structure, and suppose $m$ is a continuous nondegenerate probability measure on $\Reals$.  If 
$\mu_{(\PP,m)}$ is concentrated on 
some countable 
union
of isomorphism classes in $\Model$,
then in fact $\mu_{(\PP,m)}$ is concentrated on a single isomorphism class.
\end{corollary}
In other words, 
for any samplable Borel $L$-structure $\PP$, the measure $\mu_{(\PP,m)}$, as defined in
\S\ref{BorelLStructuresSection},
is concentrated on either one or uncountably many isomorphism
classes.
For 
an
investigation of some circumstances with
continuum-many isomorphism classes, see \cite{AFNP}.

\subsection{Continuum-sized models of Scott sentences}\ \\
\label{scottcontinuum}
\indent
We conclude with a somewhat unexpected corollary of the machinery that we
have developed.  A countable structure $\cM$ is said to be \emph{absolutely
characterizable} when its Scott sentence $\sigma_{\cM}$ has no uncountable
models, and hence characterizes $\cM$ up to isomorphism among all
structures, not just among countable structures (see
\cite[\S1.3]{MR2062240}).  Our results imply that there is no invariant measure
concentrated on such a structure.

\begin{corollary}
Let $L$ be a countable language and let $\cM\in\Model$.
Suppose that $\sigma_\cM$, the Scott sentence of $\cM$,
has no continuum-sized models. Then
there is no invariant measure on $\Model$ 
that is concentrated on
the isomorphism class of $\cM$.
\end{corollary}
\begin{proof}
Suppose there exists an invariant measure concentrated on $\cM$.
Then by Theorem~\ref{KMConditions}, $\cM$ has trivial definable
closure.
Let $\Mbar$ be the canonical structure of $\cM$ and $L_{\Mbar}$ be the
canonical language.  By Lemmas~\ref{canon-is-interdefinable} and
\ref{trivial-dcl-interdefinability}, $\Mbar$ also has trivial definable
closure.

By Proposition~\ref{pithypitwotheory-canonical}, 
there is a pithy $\Pi_2$ $\Lwow(L_\Mbar)$-theory $T_\Mbar$
all of whose countable models are isomorphic to $\Mbar$.
Hence
by Theorem~\ref{ModelsWitnessingT}
there
exists a (continuum-sized) Borel
$L_\Mbar$-structure $\QQ$ strongly witnessing $T_\Mbar$.
But then $\QQ\models T_\Mbar$, by
Lemma~\ref{SubmodelsSatisfyingT}.
By Lemma~\ref{interdefinition-lemma}, using the interdefinition
given in Lemma~\ref{canon-is-interdefinable} between $\cM$ and $\Mbar$,
there is a (continuum-sized) $L$-structure
interdefinable with $\QQ$, which has the same $\Lwow(L)$-theory as
$\cM$, and which hence satisfies
$\sigma_\cM$.
\end{proof}

Finally, this shows
that if the Scott sentence $\sigma_{\cM}$ of
a countable infinite structure $\cM$
has no continuum-sized models
(e.g., if $\cM$ is absolutely characterizable), then $\cM$ must have
nontrivial definable closure.

\vspace{10pt}

\section*{Acknowledgements}

We would like to thank Gregory Cherlin for pointing us to \cite{MR2724668}, and Lionel Levine, G{\'a}bor Lippner, Ben Rossman, Gerald Sacks, and Peter Winkler for helpful initial discussions.  We are grateful to Anatoly Vershik and Jaroslav Ne\v{s}et\v{r}il for generously sharing their insights over several conversations.  Thanks to John Truss for inviting C.\,F.\ and R.\,P.\ to the Workshop on Homogeneous Structures at the University of Leeds in 2011, and to Jaroslav Ne\v{s}et\v{r}il for inviting C.\,F.\ to the Workshop on Graph and Hypergraph Limits at the American Institute of Mathematics in 2011; both provided stimulating environments in which to think about this work.  We also thank Dan Roy for helping to improve the proofs of certain results, and Willem Fouch{\'e}, 
Alex Kruckman,
M.~Malliaris, Peter Orbanz, Max Weiss, and Carol Wood for comments on a draft.
We thank Alexander Kechris for sharing with us the argument of
Kechris--Marks incorporated into Theorem~\ref{KMConditions} and the
observation 
Corollary~\ref{KMCor}.
Finally, we would like to thank the anonymous referees for several
comments which have significantly improved the presentation of the paper.

C.\,F.\ was partially supported by NSF grants DMS-0901020 and DMS-0800198
and ARO grant W911NF-13-1-0212, and this research was partially done
while he was a visiting fellow at the Isaac Newton Institute for the
Mathematical Sciences in the program \emph{Semantics and Syntax}.  His work
on this publication was also made possible through the support of 
grants
from the John Templeton Foundation and Google. The opinions expressed in this publication are those of the authors and do not necessarily reflect the views of the John Templeton Foundation.


\bibliographystyle{amsnomr}
\begin{small}

\newcommand{\etalchar}[1]{$^{#1}$}
\def\cprime{$'$} \def\polhk#1{\setbox0=\hbox{#1}{\ooalign{\hidewidth
  \lower1.5ex\hbox{`}\hidewidth\crcr\unhbox0}}}
  \def\polhk#1{\setbox0=\hbox{#1}{\ooalign{\hidewidth
  \lower1.5ex\hbox{`}\hidewidth\crcr\unhbox0}}} \def\cprime{$'$}
  \def\cprime{$'$} \def\cprime{$'$} \def\cprime{$'$} \def\cprime{$'$}
  \def\cprime{$'$} \def\cprime{$'$} \def\cprime{$'$} \def\cprime{$'$}
\providecommand{\bysame}{\leavevmode\hbox to3em{\hrulefill}\thinspace}
\providecommand{\MR}{\relax\ifhmode\unskip\space\fi MR }
\providecommand{\MRhref}[2]{%
  \href{http://www.ams.org/mathscinet-getitem?mr=#1}{#2}
}
\providecommand{\href}[2]{#2}

\end{small}


\end{document}